\documentclass{amsart}

\usepackage{amsmath}
\usepackage{amsthm}
\usepackage{amssymb}
\usepackage[alphabetic]{amsrefs}
\usepackage{amscd}

\newcommand{ \N } [0] { \mathbf{N} }
\newcommand{ \Z } [0] { \mathbf{Z} }
\newcommand{ \Q } [0] { \mathbf{Q} }
\newcommand{ \R } [0] { \mathbf{R} }
\newcommand{ \C } [0] { \mathbf{C} }
\newcommand{ \F } [0] { \mathbf{F} }

\newcommand{ \mult } [0] { \mathbf{G}_{\textup{m}} }

\DeclareMathOperator{\kernel}{ker}
\DeclareMathOperator{\image}{im}
\DeclareMathOperator{\gal}{Gal}
\DeclareMathOperator{\frob}{Fr}
\DeclareMathOperator{\inertia}{In}
\DeclareMathOperator{\unr}{unr}
\DeclareMathOperator{\semisimple}{ss}
\DeclareMathOperator{\myhom}{Hom}
\DeclareMathOperator{\myend}{End}
\DeclareMathOperator{\myaut}{Aut}
\DeclareMathOperator{\myspan}{span}
\DeclareMathOperator{\myrank}{rank}
\DeclareMathOperator{\val}{val}
\DeclareMathOperator{\dom}{dom}
\DeclareMathOperator{\aff}{aff}
\DeclareMathOperator{\tor}{tor}
\DeclareMathOperator{\fixer}{Fix}
\DeclareMathOperator{\identity}{id}
\DeclareMathOperator{\eval}{ev}
\DeclareMathOperator{\pr}{pr}
\DeclareMathOperator{\can}{can}
\DeclareMathOperator{\inc}{inc}

\newcommand{ \canarrow } [0] { \stackrel{\can}{\longrightarrow} }
\newcommand{ \incarrow } [0] { \stackrel{\inc}{\longrightarrow} }
\newcommand{ \directedisom } [0] { \stackrel{ \sim }{ \longrightarrow } }

\newcommand{ \defeq } [0] { \stackrel{\textup{def}}{=} }

\newcommand{ \suchthat } [0] { \hspace{5pt} \vert \hspace{5pt} }
\newcommand{ \chargroup } [0] { X^{*} }
\newcommand{ \cochargroup } [0] { X_{*} }
\newcommand{ \catgroups } [0] { \textup{Groups} }

\newcommand{ \affineheckebasissymbol } [0] {\mathbb{T}}
\newcommand{ \iwahoriheckebasissymbol } [0] {\mathcal{T}}
\newcommand{ \affineheckethetaelementsymbol } [0] {\theta}
\newcommand{ \iwahoriheckethetaelementsymbol } [0] {\Theta}
\newcommand{ \normalizedaffineheckebasissymbol } [0] {\overline{\affineheckebasissymbol}}
\newcommand{ \normalizediwahoriheckebasissymbol } [0] {\overline{\iwahoriheckebasissymbol}}
\newcommand{ \finiteweylsymbol } [0] {\circ}
\newcommand{ \finiteweylgroup } [0] {W_{\finiteweylsymbol}}
\newcommand{ \apartmentsymbol } [0] {\mathcal{A}}
\newcommand{ \alcovesymbol } [0] {\mathfrak{A}}
\newcommand{ \vertexsymbol } [0] {\mathfrak{v}}
\newcommand{ \relativerootsystem } [0] {\Phi_{\finiteweylsymbol}}
\newcommand{ \scaledrootsystemsymbol } [0] {\Sigma}
\newcommand{ \param } [0] {\mathbf{q}}
\newcommand{ \indexedparam } [2] {\param_{#1}(#2)}

\newtheorem{lemma}{Lemma}[subsection]
\newtheorem{prop}{Proposition}[subsection]
\newtheorem{remark}{Remark}[subsection]

\newtheorem*{notation}{Notation}
\newtheorem*{defn}{Definition}
\newtheorem*{corollary}{Corollary}
\newtheorem*{maintheorem}{Main Theorem}

\begin{document}

\title[The Bernstein presentation in general]{The Bernstein presentation for general connected reductive groups}


\author{Sean Rostami}

\address{University of Wisconsin \\ Department of Mathematics \\ 480 Lincoln Dr. \\ Madison, WI 53706-1325}

\email{srostami@math.wisc.edu}

\email{sean.rostami@gmail.com}


\begin{abstract}
Let $ F $ be a non-Archimedean local field and let $ G $ be a connected reductive affine algebraic $ F $-group. Let $ I \subset G(F) $ be an Iwahori subgroup and denote by $ \mathcal{H} ( G ; I ) $ the Iwahori-Hecke algebra, i.e. the convolution algebra of functions $ G(F) \rightarrow \C $ which are left- and right-invariant by $ I $-translations. This article proves that the Iwahori-Hecke algebra $ \mathcal{H} ( G ; I ) $ has both an Iwahori-Matsumoto Presentation and a Bernstein Presentation analogous to those for affine Hecke algebras on root data found in \cite{lusztig}, and gives a basis (in other words, an explicit Bernstein Isomorphism) for the center $ Z [ \mathcal{H} ( G ; I ) ] $ also analogous to that found in \cite{lusztig}.
\end{abstract}

\maketitle

\tableofcontents

\section{Introduction}

Let $ F $ be a non-Archimedean local field and let $ G $ be a connected reductive affine algebraic $ F $-group. If $ J \subset G(F) $ is a compact open subgroup and $ ( \rho, V ) $ is a smooth representation of $ J $ then the \emph{Hecke algebra} $ \mathcal{H} ( G ; J, \rho ) $ is the convolution algebra of compactly-supported functions $ f : G(F) \rightarrow \myend_{\C} ( V ) $ such that $ f ( \alpha \cdot g \cdot \beta ) = \rho ( \alpha ) \circ f ( g ) \circ \rho ( \beta ) $ for all $ g \in G(F) $ and all $ \alpha, \beta \in J $.

A standard construction of supercuspidal representations of $ G(F) $ is to compactly induce certain representations $ \rho $ of compact open subgroups $ J $, and the $ G(F) $-linear endomorphism ring of such an induced representation is the Hecke algebra $ \mathcal{H} ( G ; J, \rho ) $. The Bushnell-Kutzko theory of types is a general framework and toolset to construct pairs $ ( J, \rho ) $ for which subcategories of smooth representations induced from supercuspidal representations of Levi subgroups of $ G(F) $ can be identified with categories of modules over algebras $ \mathcal{H} ( G ; J, \rho ) $. The prototype for this theory is the Borel-Casselman Theorem, which implies that the unramified principal series (the subcategory of smooth representations of $ G(F) $ whose irreducible subquotients can be found in representations induced from unramified characters on minimal Levi subgroups) is equivalent to the category of modules over an Iwahori-Hecke algebra.

Further, Hecke algebras are also the natural home for certain important functions from the arithmetic theory of Shimura varieties. For example, Pappas and Zhu \cite{PZ} recently proved \emph{Kottwitz's Conjecture}: if $ M $ is a local model (say, over $ \Z_p $) of a Shimura variety with Iwahori level-structure (with the added assumption that in the local Shimura datum $ ( G, \{ \mu \} ) $ of $ M $, the group $ G $ is unramified and the cocharacter $ \mu $ is minuscule) then the semisimple trace-of-Frobenius function on $ M ( \F_p ) $ of the nearby cycles sheaf belongs to the center of an Iwahori-Hecke algebra of $ G $. This trace function can be computed explicitly if a Bernstein Presentation (see \S1.1 below) for the Iwahori-Hecke algebra is known, due to the qualitative characterization in \cite{haines1} of the minuscule ``central Bernstein basis functions'' that the presentation yields.

To date, the Bernstein Presentations in print are due to Bernstein, Zelevinsky, and Lusztig \cite{lusztig} and apply to the affine Hecke algebra $ \mathcal{H} ( \Psi, \Delta, \param ) $ on any based reduced root datum $ ( \Psi, \Delta ) $ with any parameter system $ \param : \Delta \rightarrow \N $. These abstractly defined affine Hecke algebras help clarify the study of smooth representations because in some cases, for example if $ G $ is unramified, any Iwahori-Hecke algebra of $ G $ is naturally isomorphic to an affine Hecke algebra for appropriate choice of root datum and parameter system. Unfortunately, many Iwahori-Hecke algebras, including some of the ones considered in the theory of Shimura varieties, are not affine Hecke algebras for any choice of root datum or parameter system and so are not within the scope of the Bernstein-Zelevinsky/Lusztig work.

In this paper, we establish a Bernstein Presentation for the Iwahori-Hecke algebra of any connected reductive $ F $-group and deduce from it a description of the center.

\subsection{Summary of results}

Let $ F $ be a non-Archimedean local field with integers $ \mathcal{O}_F $ and let $ G $ be a connected reductive affine algebraic $ F $-group. Let $ A \subset G $ be a maximal $ F $-split torus, define the Levi subgroup $ M \defeq C_G(A) $, and let $ \finiteweylgroup \defeq N_G(A)(F) / M(F) $ be the corresponding relative finite Weyl group. Denote by $ \langle - , - \rangle $ the usual pairing $ \chargroup ( A ) \times \cochargroup ( A ) \rightarrow \Z $ and by $ \langle -, - \rangle_{\R} $ its natural $ \R $-bilinear extension.

Let $ I \subset G(F) $ be an Iwahori subgroup compatible with $ A $, i.e. the $ \mathcal{O}_F $-points of the connected group scheme $ \mathfrak{G}^0_{\alcovesymbol} $ attached by \cite{BT2} to an alcove $ \alcovesymbol $ in the apartment of $ A $. Denote by $ \mathcal{H} ( G; I ) $ the Iwahori-Hecke algebra, i.e. the algebra of compactly-supported functions $ f : G(F) \rightarrow \C $ satisfying $ f ( \alpha g \beta ) = f ( g ) $ for all $ g \in G(F) $, $ \alpha, \beta \in I $, with ring operation $ * $ defined by convolution relative to the Haar measure for which $ I $ is unit volume. By the Appendix to \cite{PRH}, it is known that the double-cosets $ I \backslash G(F) / I $ have the Iwahori-Weyl group $ \widetilde{W} $ of $ ( G, A ) $ as a natural system of representatives (see \S\ref{SSparahoricsandiwahoriweylgroups}). If $ w $ is such a representative then denote by $ \iwahoriheckebasissymbol_w $ the characteristic function of $ I w I $ (the collection of such functions is a $ \C $-linear basis for the Iwahori-Hecke algebra) and define $ \param : \widetilde{W} \rightarrow \N $ by $ \param ( w ) \defeq [ I w I : I ] $.

It can be shown (Proposition \ref{Paffineweylgroup}) that $ \widetilde{W} $ contains as a normal subgroup the affine Weyl group $ W_{\aff} ( \scaledrootsystemsymbol ) $ of some reduced root system $ \scaledrootsystemsymbol $ (which can be different from the relative root system of $ ( G, A ) $, although $ \finiteweylgroup $ is the same as the finite Weyl group of $ \scaledrootsystemsymbol $). Let $ \Delta_{\aff} $ be the Coxeter generating set for $ W_{\aff} ( \scaledrootsystemsymbol ) $ corresponding to $ I $ and let $ \Delta_{\finiteweylsymbol} \subset \Delta_{\aff} $ be the Coxeter generating set for $ \finiteweylgroup $. Denote by $ \ell : W_{\aff} ( \scaledrootsystemsymbol ) \rightarrow \N $ the length function induced by $ \Delta_{\aff} $. In fact, $ \widetilde{W} $ is a semidirect product of $ W_{\aff} ( \scaledrootsystemsymbol ) $ by the image $ \Omega_G $ of the Kottwitz homomorphism $ \kappa_G $ from \cite{kottwitz}, so $ \ell $ extends trivially to $ \ell : \widetilde{W} \rightarrow \N $.

Denote by $ \Omega_M $ the image of the Kottwitz homomorphism $ \kappa_M $ of $ M(F) $. It is known from \cite{HRo} that $ \widetilde{W} $ has a semidirect product decomposition $ \widetilde{W} = \Omega_M \rtimes \finiteweylgroup ( \scaledrootsystemsymbol ) $. There is a natural notion of dominance for $ \Omega_M $ (see \S\ref{SSdominance}) and if $ \lambda \in \Omega_M $ is arbitrary then there exist dominant $ \lambda_1, \lambda_2 $ such that $ \lambda = \lambda_1 - \lambda_2 $. Define, following \cite{lusztig}, $ \iwahoriheckethetaelementsymbol_{\lambda} \defeq \normalizediwahoriheckebasissymbol_{\lambda_1} * \normalizediwahoriheckebasissymbol_{\lambda_2}^{-1} $ where $ \normalizediwahoriheckebasissymbol_w \defeq \param ( w )^{-\frac{1}{2}} \iwahoriheckebasissymbol_w $.

The main result of this paper is that $ \mathcal{H} ( G; I ) $ has an \emph{Iwahori-Matsumoto Presentation} (Proposition \ref{Piwahorimatsumotopresentation}) and a \emph{Bernstein Presentation} (Propositions \ref{Pnoncentralbernsteinbasis} / \ref{Pbernsteinadditivityrelation} / \ref{Pbernsteincommutationrelation}), and that the center $ Z [ \mathcal{H} ( G; I ) ] $ has a \emph{Bernstein Isomorphism} (Propositions \ref{Pweylorbitfunctionsarecentral} / \ref{Pcentralbernsteinbasis}):
\begin{maintheorem}
The product of any two basis elements $ \iwahoriheckebasissymbol_w $ is determined by the relations \textbf{(1)} $ \iwahoriheckebasissymbol_w * \iwahoriheckebasissymbol_{w^{\prime}} = \iwahoriheckebasissymbol_{ w w^{\prime} } $ , \textbf{(2)} $ \iwahoriheckebasissymbol_s * \iwahoriheckebasissymbol_s = ( \param(s) - 1 ) \iwahoriheckebasissymbol_s + \param(s) \iwahoriheckebasissymbol_1 $, and \textbf{(3)} $ \iwahoriheckebasissymbol_w * \iwahoriheckebasissymbol_{\tau} = \iwahoriheckebasissymbol_{ w \tau } = \iwahoriheckebasissymbol_{\tau} * \iwahoriheckebasissymbol_{\tau^{-1} w \tau} $ for all $ w, w^{\prime} \in \widetilde{W} $ such that $ \ell ( w w^{\prime} ) = \ell ( w ) + \ell ( w^{\prime} ) $, all $ s \in \Delta_{\aff} $, and all $ \tau \in \Omega_G $.

The set $ \{ \iwahoriheckethetaelementsymbol_{\lambda} * \iwahoriheckebasissymbol_w \suchthat \lambda \in \Omega_M, w \in \finiteweylgroup ( \scaledrootsystemsymbol ) \} $ is a $ \C $-linear basis for the Iwahori-Hecke algebra $ \mathcal{H} ( G ; I ) $, and the product of two basis elements is determined by the relations \textbf{(1)} $ \iwahoriheckethetaelementsymbol_{\lambda} * \iwahoriheckethetaelementsymbol_{\mu} = \iwahoriheckethetaelementsymbol_{\lambda + \mu} $ and \textbf{(2)} $ \iwahoriheckebasissymbol_s * \iwahoriheckethetaelementsymbol_{\lambda} = \iwahoriheckethetaelementsymbol_{s ( \lambda )} * \iwahoriheckebasissymbol_s + \sum_{ \jmath = 0 }^{ \langle \alpha, \lambda \rangle_{\R} - 1 } \indexedparam{\jmath}{s} \iwahoriheckethetaelementsymbol_{ \lambda - \jmath \alpha^{\vee} } $ for all $ \lambda, \mu \in \Omega_M $ and all $ s \in \Delta_{\finiteweylsymbol} $ (here $ \alpha \in \scaledrootsystemsymbol $ is the root corresponding to $ s $ and $ \indexedparam{\jmath}{s}$ are defined more precisely in \S\ref{SScommutationrelation}).

Finally, the set of all elements $ z_{\mathcal{O}} \defeq \sum_{\lambda \in \mathcal{O}} \iwahoriheckethetaelementsymbol_{\lambda} $ for all $ \finiteweylgroup $-orbits $ \mathcal{O} $ in $ \Omega_M $ is a $ \C $-linear basis for the center $ Z [ \mathcal{H} ( G ; I ) ] $.
\end{maintheorem}

Upon completion of the first version of this paper, the author learned of Vign\'{e}ras work \cite{vigneras} on the \emph{pro-$ p $ Iwahori-Hecke algebra}, which also proves these results.

\subsection{Outline of paper}

In \S\ref{Snotation}, we set some notation, recall a few basic objects and functions, and state some well-known facts about them.

In \S\ref{Sgeneralitiesoniwahoriweylgroups}, we prove some basic facts concerning the Iwahori-Weyl group $ \widetilde{W} $. In order to use as much as possible in \S\ref{Sproofbernsteinrelation} the Bernstein Presentation in \cite{lusztig} for extended affine Hecke algebras, a certain based reduced root datum $ ( \Psi_M, \Delta ) $ is defined. This $ \Psi_M $ is not generally the relative root datum of $ ( G, A ) $, but for suitable choice of parameters $ \param $ the extended affine Hecke algebra $ \mathcal{H} ( \Psi_M, \Delta, \param ) $ is ``almost'' $ \mathcal{H} ( G; I ) $.

In \S\ref{Siwahorimatsumotopresentation}, we prove that the Iwahori-Hecke algebra $ \mathcal{H} ( G; I ) $ has an Iwahori-Matsumoto Presentation with parameter system $ \param : \Delta_{\aff} \rightarrow \N $ the expected index function $ s \mapsto [ I s I : I ] $. We also prove that these indices $ [ I s I : I ] $, although seemingly different from the indices that occur in a closely-related situation, are actually the same. This latter fact is not logically necessary to establish the main theorem of this paper.

In \S\ref{Sproofbernsteinrelation}, we define, following \cite{lusztig}, the elements $ \Theta_{\lambda} $ ($ \lambda \in \Omega_M $) that will be used in the Bernstein Presentation, prove the Bernstein relations (an additivity relation and a commutation relation), and prove that $ \{ \Theta_{\lambda} * \iwahoriheckebasissymbol_w \suchthat \lambda \in \Omega_M, w \in \finiteweylgroup \} $ is a basis for $ \mathcal{H} ( G; I ) $. To accomplish most of this, a certain $ \C $-algebra homomorphism $ \iota_{\mathcal{H}} : \mathcal{H} ( G; I ) \twoheadrightarrow \mathcal{H} ( \Psi_M, \Delta_{\aff}, \param ) $ is used to transport analogous known facts from \cite{lusztig} to $ \mathcal{H} ( G; I ) $. For example, because the commutation relation for $ \mathcal{H} ( \Psi_M, \Delta_{\aff}, \param ) $ is known by \cite{lusztig}, the difference between the commutator $ \iwahoriheckebasissymbol_s * \Theta_{\lambda} - \Theta_{s (\lambda)} * \iwahoriheckebasissymbol_s $ and the ``expected right-hand-side'' of the commutation relation for $ \mathcal{H} ( G; I ) $ belongs to $ \kernel ( \iota_{\mathcal{H}} ) $, and a brief inspection of that kernel shows that the difference must actually be $ 0 $.

In \S\ref{Scenter}, we define, again following \cite{lusztig}, the central Bernstein basis functions $ z_{\mathcal{O}} $ for all $ \finiteweylgroup $-orbits $ \mathcal{O} $ in $ \Omega_M $. To prove that these elements are central, the analogous known fact from \cite{lusztig} is exploited as before: the commutators $ z_{\mathcal{O}} * \iwahoriheckebasissymbol_s - \iwahoriheckebasissymbol_s * z_{\mathcal{O}} $ belong to $ \kernel ( \iota_{\mathcal{H}} ) $ and inspection of the support of $ z_{\mathcal{O}} * \iwahoriheckebasissymbol_s - \iwahoriheckebasissymbol_s * z_{\mathcal{O}} $ shows that the commutator must actually be $ 0 $. A similar proof is used for linear-independence of the set of all $ z_{\mathcal{O}} $. To prove spanning, we apply the main theorem from \cite{roro}.

In \S\ref{Scompactbernsteinrelation}, two concrete examples are given that illustrate the connection between (1) the question of whether the parameters appearing in the commutation relation are ``uniform'' or ``alternating'', (2) the construction $ \Phi_{\aff} \leadsto \scaledrootsystemsymbol $ of the scaled (or \'{e}chelonnage) root system, and (3) the relationship of the translation subgroup $ \Omega_M \subset \widetilde{W} $ to $ \cochargroup ( A ) $.

\subsection{Acknowledgements} I thank Thomas Haines for suggesting this question, for many good conversations about the subject, for his continued interest in my development, and for noticing some errors in and offering several other improvements to previous versions of this paper. I thank Shrenik Shah for alerting me to various errors and typos in the final section of the paper. Finally, I thank the referee for many nice improvements to the exposition.

\section{Setup and notation} \label{Snotation}

Let $ F $ be a non-Archimedean local field and $ G $ a connected reductive affine algebraic $ F $-group. The symbols $ \N $, $ \Z $, $ \Q $, $ \R $, and $ \C $ denote the natural numbers (including $ 0 $), the integers, the rational numbers, the real numbers, and the complex numbers. If $ H $ is any connected reductive affine algebraic $ F $-group then $ \widehat{H} $ denotes the (connected) dual $ \C $-group, i.e. the connected reductive $ \C $-group whose root datum (relative to some maximal torus) is that obtained by exchanging characters with cocharacters and roots with coroots in the root datum of $ H $. If $ \Gamma $ is any abelian group, $ \Gamma_{\tor} $ denotes the torsion subgroup. When no confusion is likely, any map that is equal to or induced by an identity map, canonical/tautological function, coordinate projection, or inclusion will be labeled by $ \identity $, $ \can $, $ \pr $, or $ \inc $.

\subsection{Fields and Galois groups}

Fix a separable closure $ F^s $ of $ F $, let $ F^{\unr} \subset F^s $ be the maximal unramified extension of $ F $, and $ L $ the completion of $ F^{\unr} $. Let $ \frob \in \gal ( F^{\unr} / F ) $ be the Frobenius automorphism, and use the same symbol $ \frob $ to refer to its continuous extension to $ L $. Similarly, let $ \inertia \defeq \gal ( F^s / F^{\unr} ) $ be the inertia subgroup and identify $ \inertia $ also with $ \gal ( L^s / L ) $. Denote by $ q $ the size of the residue field of $ F $.

\subsection{Split tori and relative roots}

Let $ A \subset S \subset T $ be $ F $-subtori of $ G $ such that $ A $ is a maximal $ F $-split torus, $ S $ is a maximal $ L $-split torus defined over $ F $, and $ T $ is a maximal torus defined over $ F $. Set $ M \defeq C_G ( A ) $. The perfect pairing $ \chargroup ( A ) \times \cochargroup ( A ) \rightarrow \Z $ is denoted $ \langle - , - \rangle $.

Let $ \relativerootsystem $ be the relative root system of $ ( G, A ) $ and $ \finiteweylgroup = N_G ( A ) / M $ the finite relative Weyl group. Denote by $ \Delta_{\finiteweylsymbol} $ a Coxeter generating set for $ \finiteweylgroup $.

If $ \Phi $ is any root system, denote by $ \finiteweylgroup ( \Phi ) $ the finite Weyl group of that root system (so, for example, $ \finiteweylgroup = \finiteweylgroup ( \relativerootsystem ) $) and by $ \Phi^{\vee} $ the coroot system.

If $ H $ is any affine algebraic $ F $-group then $ \chargroup ( H )_F $ denotes the subgroup of characters $ H \rightarrow \mult $ that are defined over $ F $.

\subsection{Apartments and affine Weyl groups}

Let $ \apartmentsymbol $ be the (enlarged) apartment for $ ( G, A ) $, denote by $ \vertexsymbol \in \apartmentsymbol $ a special vertex, and use $ \vertexsymbol $ to identify $ \apartmentsymbol = \cochargroup ( A ) \otimes_{\Z} \R $. Set $ \apartmentsymbol^{\vee} \defeq \chargroup ( A ) \otimes_{\Z} \R $. The pairing $ \langle - , - \rangle $ extends to an $ \R $-bilinear product $ \apartmentsymbol^{\vee} \times \apartmentsymbol \rightarrow \R $, which is denoted by $ \langle -, - \rangle_{\R} $. Sometimes the semisimple (or reduced) apartment $ \apartmentsymbol^{\semisimple} \defeq \myspan_{\R} ( \relativerootsystem^{\vee} ) \subset \apartmentsymbol $ is also considered.

Let $ \Phi_{\aff} $ be the set of affine roots for $ ( G, A ) $ on $ \apartmentsymbol $. Let $ \scaledrootsystemsymbol $ be the scaled (or \'{e}chelonnage) root system for $ \Phi_{\aff} $, i.e. the unique \emph{reduced} root system in $ \chargroup ( A ) \otimes_{\Z} \R $ such that the nullspaces of the affine functionals $ \scaledrootsystemsymbol + \Z $ are identical to those of $ \Phi_{\aff} $. Note that $ \finiteweylgroup = \finiteweylgroup ( \scaledrootsystemsymbol ) $ and $ \apartmentsymbol^{\semisimple} = \myspan_{\R} ( \scaledrootsystemsymbol^{\vee} ) $ also.

If $ \Phi $ is any reduced root system, denote by $ W_{\aff} ( \Phi ) $ the affine Weyl group in the sense of Ch VI \S2 no. 1 of \cite{bourbaki}, i.e. the group generated by the reflections across the nullspaces in $ \apartmentsymbol^{\semisimple} $ of the affine functionals $ \alpha + n $ for all $ \alpha \in \Phi $ and $ n \in \Z $. The affine Weyl group of $ ( G, A ) $, i.e. the group of affine transformations of $ \apartmentsymbol $ generated by the reflections across the nullspaces of $ \Phi_{\aff} $, is canonically isomorphic to the group $ W_{\aff} ( \scaledrootsystemsymbol ) $ and it is convenient to ignore the distinction. If $ \Psi = ( X, \Phi, X^{\vee}, \Phi^{\vee} ) $ is any root datum, denote by $ \widetilde{W} ( \Psi ) $ the extended affine Weyl group $ X^{\vee} \rtimes \finiteweylgroup ( \Phi ) $ of that root datum.

Let $ \mathcal{C} \subset \apartmentsymbol $ be the Weyl chamber corresponding to $ \Delta_{\finiteweylsymbol} $, i.e. the set of all $ x \in \apartmentsymbol $ such that $ \langle \alpha, x \rangle_{\R} > 0 $ for all $ \alpha \in \Delta_{\finiteweylsymbol} $. Let $ \alcovesymbol \subset \mathcal{C} $ be the alcove at $ \vertexsymbol $ and let $ \Delta_{\aff} $ be the Coxeter generating set for $ W_{\aff} ( \scaledrootsystemsymbol ) $ extended from $ \Delta_{\finiteweylsymbol} $. Note that the hyperplanes that are fixed pointwise by single elements of $ \Delta_{\aff} $ are exactly the walls of $ \alcovesymbol $.

\subsection{Irreducible subsystems} \label{SSirredsubsystems}

\begin{center}
\emph{This notation is used only in \S\ref{SSrorocompat} to establish that the main results of \cite{roro} apply to reducible root systems, despite the assumption in \cite{roro} of irreducibility.}
\end{center}

Let $ \scaledrootsystemsymbol_1, \ldots, \scaledrootsystemsymbol_r $ be the \emph{irreducible} root systems such that $ \scaledrootsystemsymbol = \scaledrootsystemsymbol_1 \cup \cdots \cup \scaledrootsystemsymbol_r $. For each $ i $, set $ \Delta_{\finiteweylsymbol}^i \defeq \Delta_{\finiteweylsymbol} \cap \finiteweylgroup ( \scaledrootsystemsymbol_i ) $ and note that $ \Delta_{\finiteweylsymbol}^i $ is a Coxeter generating set for the finite Weyl group $ \finiteweylgroup ( \scaledrootsystemsymbol_i ) $.

For each $ i $, let $ \Delta_{\aff}^i $ be the Coxeter generating set for the affine Weyl group $ W_{\aff} ( \scaledrootsystemsymbol_i ) $ extended from $ \Delta_{\finiteweylsymbol}^i $. Denote by $ \ell_i $ the length function of the Coxeter group $ ( W_{\aff} ( \scaledrootsystemsymbol_i ), \Delta_{\aff}^i ) $.

Each affine Weyl group $ W_{\aff} ( \scaledrootsystemsymbol_i ) $ acts on the corresponding ``simple'' apartment $ \apartmentsymbol^{\semisimple}_i \defeq \myspan_{\R} ( \scaledrootsystemsymbol_i^{\vee} ) \subset \apartmentsymbol $ and $ \apartmentsymbol^{\semisimple} = \apartmentsymbol^{\semisimple}_1 \times \cdots \times \apartmentsymbol^{\semisimple}_r $. Also, $ W_{\aff} ( \scaledrootsystemsymbol ) = W_{\aff} ( \scaledrootsystemsymbol_1 ) \times \cdots \times W_{\aff} ( \scaledrootsystemsymbol_r ) $ and the function $ W_{\aff} ( \scaledrootsystemsymbol_1 ) \times \cdots \times W_{\aff} ( \scaledrootsystemsymbol_r ) \rightarrow \N $ defined by $ ( w_1, \ldots, w_r ) \mapsto \ell_1 ( w_1 ) + \cdots + \ell_r ( w_r ) $ agrees with $ \ell $.

\subsection{Bruhat-Tits homomorphisms $ \nu_H $} \label{SSbruhattitshomomorphism}

For any connected reductive affine algebraic $ F $-group $ H $, there is the Bruhat-Tits homomorphism
\begin{align*}
\nu_H : H(F) &\longrightarrow \myhom_{\catgroups} ( \chargroup ( H )_F, \Z ) \\
h &\longmapsto \{ \chi \mapsto - \val_F ( \chi ( h ) ) \}
\end{align*}
Set
\begin{equation*}
H(F)^1 \defeq \kernel ( \nu_H )
\end{equation*}

Recall from \S1.2 of \cite{tits} that there is a group homomorphism
\begin{equation*}
\nu : M(F) \longrightarrow \cochargroup ( A ) \otimes_{\Z} \R
\end{equation*}
defined by post-composing $ \nu_M $ with the $ \R $-linear isomorphism
\begin{equation} \label{Ekeyvectorspaceisom1}
\myhom_{\catgroups} ( \chargroup ( M )_F, \Z ) \otimes_{\Z} \R \directedisom \cochargroup ( A ) \otimes_{\Z} \R
\end{equation}
induced by the injective finite-cokernel group homomorphism $ \chargroup ( M )_F \hookrightarrow \chargroup ( A ) $.

Note that the kernel $ \kernel ( \nu ) $ denoted by ``$ Z_c $'' in \S1.2 of \cite{tits} is exactly the same as the subgroup $ M(F)^1 $ defined above since $ \chargroup ( M )_F $ is a free-abelian group with finite rank.

As in \S1.2 of \cite{tits}, we abuse notation and denote by $ \nu $ again the extension of $ M(F) \stackrel{\nu}{\longrightarrow} \cochargroup ( A ) \otimes_{\Z} \R $ to a homomorphism from $ N_G(A)(F) $ to the group of invertible affine transformations of the apartment $ \apartmentsymbol $.

By definition, $ N_G(A)(F) $ acts on the apartment $ \apartmentsymbol $ via this $ \nu $ and, in particular, $ M(F) $ acts on $ \apartmentsymbol $ by translations. The group
\begin{equation*}
W^{\prime} \defeq N_G(A)(F) / M(F)^1
\end{equation*}
from \S1.2 of \cite{tits} acts on the apartment $ \apartmentsymbol $ via $ \nu $.

\subsection{Kottwitz homomorphisms $ \kappa_H $} \label{SSkottwitzhoms}

For any connected reductive affine algebraic $ F $-group $ H $, let $ \kappa $ be the (surjective) Kottwitz homomorphism defined for $ H $ in \S7.7 of \cite{kottwitz} and set
\begin{equation*}
\kappa_H \defeq - \kappa : H(F) \twoheadrightarrow ( \chargroup ( Z ( \widehat{H} ) )_{\inertia} )^{\frob}
\end{equation*}
Set
\begin{align*}
H(F)_1 &\defeq \kernel ( \kappa_H ) \\
\Omega_H &\defeq ( \chargroup ( Z ( \widehat{H} ) )_{\inertia} )^{\frob} \\
\overline{\Omega}_H &\defeq \Omega_H / ( \Omega_H )_{\tor}
\end{align*}

Here $ \frob $ is tacitly assumed to be lifted to $ \gal ( F^s / F ) $, and the $ \inertia $-coinvariants ensure independence from the choice of lift. The action of $ \gal ( F^s / F ) $ on the character group $ \chargroup ( Z ( \widehat{H} ) ) $ is the customary one: the natural action of $ \gal ( F^s / F ) $ on the $ F $-group $ H $ induces an action on the root datum $ \Psi $ of $ H $ (relative to some maximal $ F $-torus) and therefore tautologically on the root datum $ \widehat{\Psi} $ of $ \widehat{H} $, and $ \chargroup ( Z ( \widehat{H} ) ) $ acquires its $ \gal ( F^s / F ) $-action via its representation as the quotient of the character group of $ \widehat{\Psi} $ by the root lattice of $ \widehat{\Psi} $.

\subsection{Relationship between $ \kappa_H $ and $ \nu_H $} \label{SSkottwitzcompatwithbruhattits}

There is a group homomorphism
\begin{equation*}
q_H : \Omega_H \longrightarrow \myhom_{\catgroups} ( \chargroup ( H )_F, \Z )
\end{equation*}
such that $ \kernel ( q_H ) = ( \Omega_H )_{\tor} $. The precise definition of $ q_H $ is not necessary here, but $ q_H $ is essentially (7.4.4) in \cite{kottwitz} after applying the $ \frob $-fixed functor (to see that the codomain of (7.4.4) in \cite{kottwitz} is the same as the codomain of $ q_H $ here, one must use character/cocharacter duality for the torus $ Z ( \widehat{H} )^{\circ} $ and the isomorphism $ \chargroup ( H ) \directedisom \cochargroup ( Z ( \widehat{H} ) ) $).

By \S7.4 of \cite{kottwitz},
\begin{equation*}
q_H \circ \kappa_H = \nu_H
\end{equation*}
(note that $ \kappa_H $ and $ \nu_H $ here \emph{both} differ in sign from those used in \cite{kottwitz}) and therefore
\begin{equation*}
H(F)_1 \subset H(F)^1
\end{equation*}
Combining with $ \kernel ( q_H ) = ( \Omega_H )_{\tor} $ implies that $ \kappa_H $ induces an isomorphism
\begin{equation*}
H(F)^1 / H(F)_1 \directedisom ( \Omega_H )_{\tor}
\end{equation*}

Similar to (\ref{Ekeyvectorspaceisom1}), the map $ q_H $ induces an $ \R $-linear isomorphism
\begin{equation} \label{Ekeyvectorspaceisom2}
\Omega_H \otimes_{\Z} \R \directedisom \myhom_{\catgroups} ( \chargroup ( H )_F, \Z ) \otimes_{\Z} \R
\end{equation}

In the case $ H = M $, the isomorphism (\ref{Ekeyvectorspaceisom2}) has a more direct description: the natural homomorphism $ \cochargroup ( T ) = \chargroup ( \widehat{T} ) \rightarrow \chargroup ( Z ( \widehat{M} ) ) $ induces an injective finite-cokernel group homomorphism $ ( \cochargroup ( T )_{\inertia} )^{\frob} \rightarrow \Omega_M $ and applying $ - \otimes_{\Z} \R $ to
\begin{equation*}
\cochargroup ( A ) \incarrow ( \cochargroup ( T )_{\inertia} )^{\frob} \rightarrow \Omega_M
\end{equation*}
yields the inverse of $ (\ref{Ekeyvectorspaceisom1}) \circ (\ref{Ekeyvectorspaceisom2}) : \Omega_M \otimes_{\Z} \R \directedisom \cochargroup ( A ) \otimes_{\Z} \R $. In particular, $ \overline{\Omega}_M $ is a free-abelian group and $ \myrank ( \overline{\Omega}_M ) = \dim ( A ) $.

\subsection{Parahoric subgroups and Iwahori-Weyl groups} \label{SSparahoricsandiwahoriweylgroups}

For any subset $ S \subset \apartmentsymbol $, denote by $ \fixer_{G(F)} ( S ) $ the subset of $ g \in G(F) $ for which $ g \cdot x = x $ for all $ x \in S $.

Using the viewpoint from the Appendix to \cite{PRH}, define the Iwahori subgroup
\begin{equation*}
I \defeq G(F)_1 \cap \fixer_{G(F)} ( \alcovesymbol )
\end{equation*}
Similarly, define $ K \defeq G(F)_1 \cap \fixer_{G(F)} ( \vertexsymbol ) $. Two compact open subgroups which are closely-related to, but generally distinct from, $ I $ and $ K $ are $ \widetilde{I} \defeq G(F)^1 \cap \fixer_{G(F)} ( \alcovesymbol ) $ and $ \widetilde{K} \defeq G(F)^1 \cap \fixer_{G(F)} ( \vertexsymbol ) $.

Let $ \widetilde{W} $ denote the Iwahori-Weyl group of $ ( G, A ) $, i.e. the group
\begin{equation*}
\widetilde{W} \defeq N_G(S)(F) / T(F)_1
\end{equation*}
occurring in Remark 9 of the Appendix to \cite{PRH}. By Proposition 8 / Remark 9 of the Appendix to \cite{PRH}, any system of representatives in $ N_G(S)(F) $ for $ \widetilde{W} $ is also a system of representatives for the double-cosets $ I \backslash G(F) / I $. A representative of $ w \in \widetilde{W} $ is sometimes, by abuse of notation, also denoted by $ w $.

Since any $ F $-point of $ N_G(S) $ must normalize the maximal $ F $-split subtorus $ A \subset S $, $ N_G(S)(F) \subset N_G(A)(F) $. Since $ * \mapsto \kappa_{*} $ is functorial,
\begin{equation*}
T(F)_1 \subset M(F)_1 \subset M(F)^1
\end{equation*}
and so there is a natural group homomorphism
\begin{equation*}
\widetilde{W} \rightarrow N_G(A)(F) / M(F)^1 = W^{\prime}
\end{equation*}
By definition, the Iwahori-Weyl group $ \widetilde{W} $ acts on the apartment $ \apartmentsymbol $ by the composition of this homomorphism and $ \nu $ from \S\ref{SSbruhattitshomomorphism}.

\subsection{Semidirect product decompositions of Iwahori-Weyl groups} \label{SSsemidirectproductdecomps}

The alcove $ \alcovesymbol $ is contained within a unique alcove $ C $ in the Bruhat-Tits building for $ G(L) $, and this alcove $ C $ is necessarily $ \frob $-stable. Taking $ \frob $-fixed points of the external semidirect product provided by Lemma 14 of the Appendix to \cite{PRH} (relative to the alcove $ C $) yields the external semidirect product
\begin{equation} \label{Eiwahorimatsumotosemidirect}
\widetilde{W} \cong W_{\aff} \rtimes \Omega_G
\end{equation}
The surjection $ \widetilde{W} \twoheadrightarrow \Omega_G $ here is defined by fixing arbitrary representatives $ g_w \in N_G(S)(F) $ for each $ w \in \widetilde{W} $ and defining $ w \mapsto \kappa_G ( g_w ) $ (see Lemma 14 of the Appendix to \cite{PRH} and \S2.3 of \cite{HRo} for more details). It is worth emphasizing here that $ W_{\aff} $ is defined as merely the $ \frob $-fixed subgroup of the affine Weyl group occurring in Lemma 14 of the Appendix to \cite{PRH}, although it is verified in \S\ref{SStrueaffineweylgroup} that $ W_{\aff} $ is, in fact, the affine Weyl group of $ \scaledrootsystemsymbol $.

Assuming in advance that $ W_{\aff} = W_{\aff} ( \scaledrootsystemsymbol ) $, denote by $ \ell : W_{\aff} \rightarrow \N $ the usual length function of the Coxeter group $ ( W_{\aff}, \Delta_{\aff} ) $. Similarly, denote by $ \preceq $ the usual Bruhat-Chevalley partial order on $ W_{\aff} $ relative to $ \Delta_{\aff} $ (as usual, ``$ w_1 \prec w_2 $'' simply means ``$ w_1 \preceq w_2 $ and $ w_1 \neq w_2 $''). Extend $ \ell $ to $ \widetilde{W} $ by inflating along the projection $ \widetilde{W} \rightarrow W_{\aff} $. Extend the Bruhat-Chevalley order $ \preceq $ to $ \widetilde{W} $ by defining $ ( w_1, \tau_1 ) \preceq ( w_2, \tau_2 ) $ if and only if $ w_1 \preceq w_2 $ in $ W_{\aff} $ and $ \tau_1 = \tau_2 $ in $ \Omega_G $.

We frequently use the fact that $ I u I v I = I u v I $ for any $ u, v \in \widetilde{W} $ satisfying $ \ell ( u v ) = \ell ( u ) + \ell ( v ) $, which holds because of 5.2.12 of \cite{BT2}. This fact is referred to as the \emph{Tits System Axiom}.

Define
\begin{equation*}
\Omega_G^{\prime} = \image ( \Omega_G \stackrel{(\ref{Eiwahorimatsumotosemidirect})}{\hookrightarrow} \widetilde{W} ).
\end{equation*}
Necessarily, $ \Omega_G^{\prime} $ consists of elements that stabilize the base alcove $ \alcovesymbol $ (see Lemma 14 of the Appendix to \cite{PRH} and 5.1.28 / 5.1.14 of \cite{BT2}). For any $ w \in \widetilde{W} $, denote by $ \Omega_G ( w ) $ the projection of $ w $ onto $ \Omega_G^{\prime} $. A simple but important fact is that this projection onto $ \Omega_G^{\prime} $ is constant on conjugacy classes since $ \Omega_G $ is abelian.

Another external semidirect product is provided by Lemma 3.0.1(III) of \cite{HRo}:
\begin{equation*}
\widetilde{W} \cong \Omega_M \rtimes \finiteweylgroup
\end{equation*}
Recall from Lemma 3.0.1(III) of \cite{HRo} that, since $ M $ is anisotropic-mod-center, the apartment for $ ( M, A ) $ has only one alcove and therefore the surjection (via the analogue of (\ref{Eiwahorimatsumotosemidirect}) for $ M $) of $ N_M(S)(F) / T(F)_1 $ (which is the Iwahori-Weyl group of $ ( M, A ) $) onto $ \Omega_M $ is actually an isomorphism. The injection $ \Omega_M \hookrightarrow \widetilde{W} $ in the above semidirect product is the composition $ \Omega_M \directedisom N_M(S)(F) / T(F)_1 \subset \widetilde{W} $. Note that the restriction of the projection $ \widetilde{W} \rightarrow \Omega_M $ to the subgroup $ M(F) / M(F)_1 $ is simply $ \kappa_M $.

If $ \lambda \in \Omega_M $ then the symbol $ t_{\lambda} $ is sometimes used to distinguish $ \lambda $ from its image in $ \widetilde{W} $ or its transformation of $ \apartmentsymbol $. Accordingly, $ w ( \lambda ) $ and $ w \circ t_{\lambda} \circ w^{-1} $ have the same meaning for all $ w \in \finiteweylgroup $.

\subsection{Iwahori-Hecke algebras} \label{SSdefsofheckealgebras}

Denote by $ dG_I $ the Haar measure on $ G(F) $ for which $ dG_I(I)=1 $ and by $ \mathcal{H} ( G ; I ) $ the $ \C $-vector space of compactly-supported functions $ f : G(F) \rightarrow \C $ such that $ f ( \alpha g \beta ) = f ( g ) $ for all $ g \in G(F) $ and all $ \alpha, \beta \in I $. Make $ \mathcal{H} ( G ; I ) $ a $ \C $-algebra using the Haar measure $ dG_I $ to define convolution $ * $.

By Proposition 8 / Remark 9 of the Appendix to \cite{PRH}, the characteristic functions
\begin{equation*}
\iwahoriheckebasissymbol_w \defeq 1_{I w I} : G(F) \rightarrow \C
\end{equation*}
are a $ \C $-linear basis for the vector space $ \mathcal{H} ( G ; I ) $. If $ h \in \mathcal{H} ( G ; I ) $ and $ h = \sum_i c_i \iwahoriheckebasissymbol_{w_i} $ ($ c_i \in \C $) then $ w_j $ is said to \emph{support} $ h $ iff $ c_j \neq 0 $.

For $ w \in \widetilde{W} $, denote by $ [ I w I : I ] $ the number of distinct right cosets $ x I $ whose union is $ I w I $ (here $ w $ and $ x $ actually refer to representatives in $ N_G(S)(F) $) and define $ \param : \widetilde{W} \rightarrow \N $ by
\begin{equation*}
\param ( w ) \defeq [ I w I : I ]
\end{equation*}
Assuming in advance that $ W_{\aff} $ is the affine Weyl group of $ \scaledrootsystemsymbol $, which is proved in \S\ref{SStrueaffineweylgroup}, the Tits System Axiom implies that if $ w \in W_{\aff} $ and $ w = s_1 s_2 \cdots s_r $ is a reduced expression from $ \Delta_{\aff} $ then $ \param ( w ) = \param ( s_1 ) \param ( s_2 ) \cdots \param ( s_r ) $. If $ \tau \in \Omega_G $ then $ \tau $ (or rather its representative in $ N_G(S)(F) $) normalizes $ I $ and therefore $ \param ( \tau ) = 1 $, so in fact $ \param $ factors through $ W_{\aff} $.

\section{Generalities on Iwahori-Weyl groups} \label{Sgeneralitiesoniwahoriweylgroups}

\subsection{An alternate quotient for the Iwahori-Weyl group}

\begin{center}
\emph{Recall from \S\ref{SSparahoricsandiwahoriweylgroups} that the Iwahori-Weyl group $ \widetilde{W} $ is, a priori, the quotient $ N_G(S)(F) / T(F)_1 $. The main purpose of this subsection is to show that $ \widetilde{W} $ is also the quotient $ N_G(A)(F) / M(F)_1 $.}
\end{center}

The inclusion $ N_G(S)(F) \subset N_G(A)(F) $ yields a natural group homomorphism
\begin{equation*}
i : \widetilde{W} \longrightarrow N_G(A)(F) / M(F)_1
\end{equation*}
In fact, this map identifies the two groups:

\begin{lemma} \label{Liwahoriweylasquotient}
If $ \phi $ and $ \psi $ denote the group homomorphisms in the short-exact-sequence $ 0 \rightarrow \Omega_M \rightarrow \widetilde{W} \rightarrow \finiteweylgroup \rightarrow 1 $ from \S\ref{SSparahoricsandiwahoriweylgroups} then the diagram
\begin{equation*}
\begin{CD}
0 @>>> \Omega_M @> \phi >> \widetilde{W} @> \psi >> \finiteweylgroup @>>> 1 \\
@. @A \kappa_M AA @VV i V @A \can AA @. \\
1 @>>> \frac{M(F)}{M(F)_1} @> \inc >> \frac{N_G(A)(F)}{M(F)_1} @> \can >> \frac{N_G(A)(F)}{M(F)} @>>> 1
\end{CD}
\end{equation*}
commutes and $ i $ is an isomorphism.
\end{lemma}

\begin{proof}
Commutativity of the left square follows from the description in \S\ref{SSsemidirectproductdecomps} of the semidirect product: for any $ \lambda \in \Omega_M $ there is $ m \in N_M(S)(F) $ such that $ \kappa_M ( m ) = \lambda $ and by definition $ i ( \phi ( \lambda ) ) = m M(F)_1 $, which visibly agrees with $ \inc \circ \kappa_M^{-1} $. The right square commutes because, by the statement and proof of Lemma 6.1.2 of \cite{HRo} and from the discussion in \S6.2 of \cite{HRo}, the map $ \psi : \widetilde{W} \rightarrow \finiteweylgroup $ is precisely the homomorphism $ N_G(S)(F) / T(F)_1 \rightarrow N_G(A)(F) / M(F) $ induced by inclusions. That $ i $ is an isomorphism follows from the ``Five Lemma''.
\end{proof}

\begin{remark}
The identification yielded by Lemma \ref{Liwahoriweylasquotient} was proved independently by Timo Richarz \cite{richarz}.
\end{remark}

\begin{notation}
From now on, suppress $ i $ and identify $ \widetilde{W} = N_G(A)(F) / M(F)_1 $. It is safe to do this since $ i $ is induced by the inclusion $ N_G(S)(F) \subset N_G(A)(F) $.
\end{notation}

\subsection{The affine Weyl group inside the Iwahori-Weyl group} \label{SStrueaffineweylgroup}

\begin{center}
\emph{Recall from \S\ref{SSparahoricsandiwahoriweylgroups} that, despite the name, $ W_{\aff} $ is merely a certain $ \frob $-fixed subgroup of the Iwahori-Weyl group of $ ( G_L, S ) $. In this subsection, it is verified that $ W_{\aff} $ really is the affine Weyl group of $ ( G, A ) $, i.e. the group of transformations of $ \apartmentsymbol $ generated by all reflections across the nullspaces of $ \Phi_{\aff} $.}
\end{center}

By the main result of the Appendix to \cite{PRH} and Lemma 8.0.1 / Proposition 11.1.4 of \cite{HRo} the following logic holds:
\begin{equation*}
\left\{ \Omega_G \text{ torsion-free} \right\} \stackrel{\S\ref{SSkottwitzcompatwithbruhattits}}{\Longleftrightarrow} \left\{ G(F)^1 = G(F)_1 \right\} \stackrel{\text{8.0.1}}{\Longrightarrow} \left\{ \widetilde{K} = K \right\} \stackrel{\text{11.1.4}}{\Longleftrightarrow} \left\{ \Omega_M \text{ torsion-free} \right\}
\end{equation*}

\begin{remark}
The previous logic is also valid for intermediate parahorics $ I \subset J \subset K $ due to Lemma \ref{Lgeneralizedtorsionparam} below.
\end{remark}

This suggests the following lemma:
\begin{lemma} \label{LMtorsioninsideGtorsion}
$ M(F) \subset G(F) $ induces an injection $ ( \Omega_M )_{\tor} \hookrightarrow ( \Omega_G )_{\tor} $ via $ \kappa_M $ and $ \kappa_G $.
\end{lemma}

\begin{proof}
By Proposition 11.1.4 of \cite{HRo}, $ M(F)^1 \subset \widetilde{K} $ and induces a bijection $ M(F)^1 / M(F)_1 \directedisom \widetilde{K} / K $. On the other hand, Lemma 8.0.1 of \cite{HRo} and the main result of the Appendix to \cite{PRH} imply that $ \widetilde{K} \cap G(F)_1 = K $, so $ \widetilde{K} \subset G(F)^1 $ induces an injection $ \widetilde{K} / K \hookrightarrow G(F)^1 / G(F)_1 $. By \S\ref{SSkottwitzcompatwithbruhattits}, $ \kappa_M $ and $ \kappa_G $ identify $ M(F)^1 / M(F)_1 \cong ( \Omega_M )_{\tor} $ and $ G(F)^1 / G(F)_1 \cong ( \Omega_G )_{\tor} $, so we have an injection $ ( \Omega_M )_{\tor} \hookrightarrow ( \Omega_G )_{\tor} $. Since the composite $ M(F)^1 / M(F)_1 \hookrightarrow G(F)^1 / G(F)_1 $ of the two maps above agrees with the inclusion $ M(F) \subset G(F) $, functoriality of $ * \mapsto \kappa_{*} $ means that the injection $ ( \Omega_M )_{\tor} \hookrightarrow ( \Omega_G )_{\tor} $ is indeed induced by the inclusion.
\end{proof}

\begin{remark}
Lemma \ref{LMtorsioninsideGtorsion} appeared earlier in the Appendix of \cite{haines2}.
\end{remark}

Recall from \S\ref{SSbruhattitshomomorphism} the group $ W^{\prime} \defeq N_G(A)(F) / M(F)^1 $. The composition $ \widetilde{W} = N_G(A)(F) / M(F)_1 \rightarrow N_G(A)(F) / M(F)^1 $ is a surjective group homomorphism
\begin{equation*}
\iota : \widetilde{W} \twoheadrightarrow W^{\prime}
\end{equation*}
It is trivial from \S\ref{SSparahoricsandiwahoriweylgroups} that the action of $ \widetilde{W} $ on $ \apartmentsymbol $ factors through $ W^{\prime} $.

Denote by $ \Omega^{\prime} \subset W^{\prime} $ the subgroup of all elements which stabilize the base alcove $ \alcovesymbol $. By \S1.7 of \cite{tits}, $ W_{\aff} ( \scaledrootsystemsymbol ) $ is considered as a normal subgroup of $ W^{\prime} $ and it is completely formal that $ \Omega^{\prime} \directedisom W^{\prime} / W_{\aff} ( \scaledrootsystemsymbol ) $. The main claim of this subsection is:
\begin{prop} \label{Paffineweylgroup}
$ \iota ( \Omega_G^{\prime} ) = \Omega^{\prime} $ and the restriction of $ \iota $ to $ W_{\aff} $ is an isomorphism $ W_{\aff} \directedisom W_{\aff} ( \scaledrootsystemsymbol ) $.
\end{prop}

\begin{proof}
That $ \iota ( \Omega_G^{\prime} ) \subset \Omega^{\prime} $ is immediate since the actions of $ \widetilde{W} $ and $ W^{\prime} $ on $ \apartmentsymbol $ are compatible via $ \iota $. A simple ``Five Lemma''-style diagram chase shows that $ \iota ( W_{\aff} ) \subset W_{\aff} ( \scaledrootsystemsymbol ) $. We now verify injectivity of $ \iota \vert_{W_{\aff}} $. Since Lemma \ref{Liwahoriweylasquotient} says that $ M(F) \subset N_G(A)(F) $ is compatible via $ \kappa_M $ with $ \Omega_M \hookrightarrow \widetilde{W} $, we may identify $ \kernel ( \iota ) = ( \Omega_M )_{\tor} $. Lemma \ref{LMtorsioninsideGtorsion} and the fact that $ W_{\aff} \cap \Omega_G^{\prime} = \{ 1 \} $ then imply that $ W_{\aff} \cap \kernel ( \iota ) = \{ 1 \} $. That the restrictions $ W_{\aff} \rightarrow W_{\aff} ( \scaledrootsystemsymbol ) $ and $ \Omega_G^{\prime} \rightarrow \Omega^{\prime} $ of $ \iota $ are surjective is immediate due to the fact that $ \iota $ itself is surjective and, by what is already proven above, $ \iota $ is the product of its restrictions $ W_{\aff} \rightarrow W_{\aff} ( \scaledrootsystemsymbol ) $ and $ \Omega_G^{\prime} \rightarrow \Omega^{\prime} $.
\end{proof}

\begin{notation}
From now on, suppress the isomorphism $ \iota \vert_{W_{\aff}} : W_{\aff} \directedisom W_{\aff} ( \scaledrootsystemsymbol ) $ and identify $ W_{\aff} = W_{\aff} ( \scaledrootsystemsymbol ) $. Accordingly, $ \Delta_{\aff} $ will be considered a Coxeter generating set for $ W_{\aff} $ and $ \ell $ will denote the corresponding length function.
\end{notation}

\subsection{A genuine root datum using $ \Omega_M $} \label{SSgenuinerootdatum}

\begin{center}
\emph{The purpose of this subsection is to define a reduced root datum $ \Psi_M $ whose extended affine Hecke algebra is very nearly the Iwahori-Hecke algebra $ \mathcal{H} ( G, I ) $. This formalism allows the results of \cite{lusztig} to be used to the greatest extent.}
\end{center}

By Proposition \ref{Paffineweylgroup}, $ W_{\aff} $ contains the group of translations by $ \scaledrootsystemsymbol^{\vee} $ and since $ \Omega_M $ is the translation subgroup of $ \widetilde{W} $, it is true that $ \scaledrootsystemsymbol^{\vee} \subset \Omega_M $. Define
\begin{equation*}
\overline{\scaledrootsystemsymbol}^{\vee} \defeq \image ( \scaledrootsystemsymbol^{\vee} \subset \Omega_M \canarrow \overline{\Omega}_M )
\end{equation*}
Notice that $ \scaledrootsystemsymbol^{\vee} $ and $ \overline{\scaledrootsystemsymbol}^{\vee} $ are identical root systems.

Recall from \S\ref{SSkottwitzcompatwithbruhattits} that there is a natural map $ \cochargroup ( A ) \rightarrow \Omega_M $ and that applying $ - \otimes_{\Z} \R $ to $ \cochargroup ( A ) \rightarrow \Omega_M \canarrow \overline{\Omega}_M $ yields an $ \R $-linear isomorphism $ \cochargroup ( A ) \otimes_{\Z} \R \directedisom \overline{\Omega}_M \otimes_{\Z} \R $. Via this isomorphism, define $ \overline{\Omega}_M^{\vee} \subset \chargroup ( A ) \otimes_{\Z} \R $ to be the $ \Z $-dual of $ \overline{\Omega}_M \hookrightarrow \cochargroup ( A ) \otimes_{\Z} \R $ with respect to $ \langle - , - \rangle_{\R} $, i.e.
\begin{equation*}
\overline{\Omega}_M^{\vee} \defeq \{ x \in \chargroup ( A ) \otimes_{\Z} \R \suchthat \langle x , \overline{\Omega}_M \rangle_{\R} \subset \Z \}
\end{equation*}
Then $ \overline{\Omega}_M^{\vee} $ is also free-abelian and finite-rank and the restriction
\begin{equation*}
\langle - , - \rangle_M : \overline{\Omega}_M^{\vee} \times \overline{\Omega}_M  \longrightarrow \Z
\end{equation*}
of $ \langle - , - \rangle_{\R} $ is a perfect pairing. Further, if $ \overline{\scaledrootsystemsymbol} $ denotes the dual root system of $ \overline{\scaledrootsystemsymbol}^{\vee} $ then, by construction, $ \overline{\scaledrootsystemsymbol} \subset \overline{\Omega}_M^{\vee} $.

In summary,
\begin{lemma}
The tuple $ ( \overline{\Omega}_M^{\vee}, \overline{\scaledrootsystemsymbol}, \overline{\Omega}_M, \overline{\scaledrootsystemsymbol}^{\vee}, \langle - , - \rangle_M ) $ is a reduced root datum in the traditional sense.
\end{lemma}

\begin{defn}
The \emph{scaled} (or \emph{\'{e}chelonnage}) \emph{root datum} of $ ( G, A ) $ is the reduced root datum
\begin{equation*}
\Psi_M \defeq ( \overline{\Omega}_M^{\vee}, \overline{\scaledrootsystemsymbol}, \overline{\Omega}_M, \overline{\scaledrootsystemsymbol}^{\vee}, \langle - , - \rangle_M )
\end{equation*}
For compatibility with $ G $, it is intended that $ \overline{\Omega}_M^{\vee} $ play the role of the \emph{character} group in this root datum.
\end{defn}

\begin{notation}
Since $ \overline{\scaledrootsystemsymbol} = \scaledrootsystemsymbol $, abuse notation and identify $ \finiteweylgroup ( \overline{\scaledrootsystemsymbol} ) = \finiteweylgroup $.
\end{notation}

\subsection{Some compatibilities} \label{SSvariouscompatibilities}

\begin{center}
\emph{The purpose of this subsection is merely to record the isomorphisms between and compatibilities among the many different objects in play, and to emphasize that the extended affine Weyl group of the scaled root datum $ \Psi_M $ is just the group $ W^{\prime} $ from \cite{tits}.}
\end{center}

Denote by $ \varphi $ the isomorphism $ \widetilde{W} \directedisom \Omega_M \rtimes \finiteweylgroup $ from \S\ref{SSsemidirectproductdecomps}.
\begin{lemma} \label{Lcompatibilitysquare}
There is a unique isomorphism
\begin{equation*}
\varphi^{\prime} : W^{\prime} \defeq N_G(A)(F) / M(F)^1 \directedisom \overline{\Omega}_M \rtimes \finiteweylgroup \defeq \widetilde{W} ( \Psi_M )
\end{equation*}
such that the square
\begin{equation*}
\begin{CD}
\widetilde{W} @> \iota >> W^{\prime} \\
@V \varphi VV @VV \varphi^{\prime} V \\
\Omega_M \rtimes \finiteweylgroup @>> \can \times \identity > \widetilde{W} ( \Psi_M )
\end{CD}
\end{equation*}
commutes and which restricts to an isomorphism $ M(F) / M(F)^1 \directedisom \overline{\Omega}_M $.
\end{lemma}

Recall from \S1.2 of \cite{tits} that subgroup of $ W^{\prime} $ which acts on $ \apartmentsymbol $ by translations is the quotient $ M(F) / M(F)^1 $ and is denoted ``$ \Lambda $'' there.

\begin{proof}
If it were true that $ \widetilde{W} \stackrel{\pr \circ \varphi}{\longrightarrow} \Omega_M \canarrow \overline{\Omega}_M $ and $ \widetilde{W} \stackrel{\pr \circ \varphi}{\longrightarrow} \finiteweylgroup $ were both constant on fibers of $ \iota $ then the desired diagram itself provides the definition of $ \varphi^{\prime} $, and uniqueness is automatic. Constancy of the latter map is obvious since, as remarked in the proof of Lemma \ref{Liwahoriweylasquotient}, the map $ \widetilde{W} \rightarrow \finiteweylgroup $ is the canonical one $ N_G(A)(F) / M(F)_1 \canarrow N_G(A)(F) / M(F) $. For constancy of the former map, suppose $ n_1, n_2 \in N_G(A)(F) $ and $ \iota ( n_1 ) = \iota ( n_2 ) $, i.e. that $ n_1 n_2^{-1} \in M(F)^1 $. Since $ \varphi $ agrees with $ \kappa_M $ on $ M(F) $ (see \S\ref{SSparahoricsandiwahoriweylgroups}) and since $ \kappa_M ( M(F)^1 ) = ( \Omega_M )_{\tor} $, the images of $ n_1 M(F)_1 $ and $ n_2 M(F)_1 $ in $ \overline{\Omega}_M \rtimes \finiteweylgroup $ are indeed the same. As mentioned above, this constancy yields both the isomorphism $ \varphi^{\prime} $, commutativity of the square, and the uniqueness statement. It is automatic from the commutativity of the diagram that $ \varphi^{\prime} $ is surjective. Because $ \varphi $ is injective and agrees with $ \kappa_M $ on $ M(F) / M(F)_1 $, $ \varphi^{-1} ( x, 1 ) \in M(F)^1 / M(F)_1 $ for all $ x \in ( \Omega_M )_{\tor} $ and therefore $ \varphi^{\prime} $ is injective. The remaining facts also follow from the injectivity of $ \varphi $ and the agreement with $ \kappa_M $: $ \varphi^{\prime} $ indeed restricts to $ M(F) / M(F)^1 \rightarrow \overline{\Omega}_M $, injectivity of the restriction is automatic from that of $ \varphi^{\prime} $ itself, and the restriction is surjective because $ \kappa_M $ and $ \Omega_M \canarrow \overline{\Omega}_M $ are both surjective.
\end{proof}

Let $ \varphi^{\prime} $ be as in Lemma \ref{Lcompatibilitysquare}. Define
\begin{align*}
\overline{\Omega} &\defeq \varphi^{\prime} ( \Omega^{\prime} ) \\
W_{\aff} ( \overline{\scaledrootsystemsymbol} ) &\defeq \varphi^{\prime} ( W_{\aff} ( \scaledrootsystemsymbol ) )
\end{align*}
and note that $ \widetilde{W} ( \Psi_M ) $ inherits the internal semidirect product decomposition
\begin{equation*}
\widetilde{W} ( \Psi_M ) = W_{\aff} ( \overline{\scaledrootsystemsymbol} ) \rtimes \overline{\Omega}
\end{equation*}
from $ W^{\prime} $.

\begin{notation}
Abuse notation and denote by
\begin{equation} \label{Enewiota}
\iota : \widetilde{W} \longrightarrow \widetilde{W} ( \Psi_M )
\end{equation}
the (surjective) composition $ \widetilde{W} \stackrel{\iota}{\longrightarrow} W^{\prime} \stackrel{\varphi^{\prime}}{\longrightarrow} \widetilde{W} ( \Psi_M ) $.
\end{notation}

\begin{lemma} \label{Laffineweylisomtosyntheticaffineweyl}
$ \iota $ from (\ref{Enewiota}) restricts to an isomorphism $ W_{\aff} \directedisom W_{\aff} ( \overline{\scaledrootsystemsymbol} ) $ and $ \iota ( \Omega_G^{\prime} ) = \overline{\Omega} $.
\end{lemma}

\begin{proof}
Both are automatic by definition of $ W_{\aff} ( \overline{\scaledrootsystemsymbol} ) $ and $ \overline{\Omega} $ because of Proposition \ref{Paffineweylgroup}.
\end{proof}

\begin{notation}
From now on, consider $ \Delta_{\aff} $ as a Coxeter generating set for $ W_{\aff} ( \overline{\scaledrootsystemsymbol} ) $ via $ \iota $. Similarly, use $ \ell $ to refer to the length function on the Coxeter group $ ( W_{\aff} ( \overline{\scaledrootsystemsymbol} ), \Delta_{\aff} ) $ and its inflation along the projection $ \widetilde{W} ( \Psi_M ) \rightarrow W_{\aff} ( \overline{\scaledrootsystemsymbol} ) $.
\end{notation}

\begin{corollary}
$ \ell ( \iota ( w ) ) = \ell ( w ) $ for all $ w \in \widetilde{W} $.
\end{corollary}

\begin{proof}
This is immediate from Lemma \ref{Laffineweylisomtosyntheticaffineweyl} because the length functions were extended trivially from the same based affine Weyl group.
\end{proof}

\section{The Iwahori-Matsumoto presentation} \label{Siwahorimatsumotopresentation}

\subsection{Statement and proof}

Recall from \S3.2 that $ W_{\aff} $ is a Coxeter group with Coxeter generating set $ \Delta_{\aff} $. The following, which is essentially just Exercise 24 in Chapter IV, \S2 of \cite{bourbaki}, gives an Iwahori-Matsumoto Presentation for $ \mathcal{H} ( G ; I ) $:
\begin{prop} \label{Piwahorimatsumotopresentation}
For all $ w \in W_{\aff} $, all reduced expressions $ w = s_1 \cdots s_n $ ($ s_i \in \Delta_{\aff} $), all $ z \in \Omega_G $, and all $ s \in \Delta_{\aff} $ the following identities are true in $ \mathcal{H} ( G ; I ) $:
\begin{enumerate}
\item \label{IMPnormalizer} $ \iwahoriheckebasissymbol_{w} * \iwahoriheckebasissymbol_{z} = \iwahoriheckebasissymbol_{ w z } = \iwahoriheckebasissymbol_{z} * \iwahoriheckebasissymbol_{z^{-1} w z} $

\item \label{IMPbraid} $ \iwahoriheckebasissymbol_w = \iwahoriheckebasissymbol_{s_1} * \cdots * \iwahoriheckebasissymbol_{s_n} $

\item \label{IMPquadratic} $ \iwahoriheckebasissymbol_s * \iwahoriheckebasissymbol_s = ( \param(s) - 1 ) \iwahoriheckebasissymbol_s + \param(s) \iwahoriheckebasissymbol_1 $
\end{enumerate}
\end{prop}

\begin{proof}
We first prove (\ref{IMPnormalizer}). By examining the integral that defines $ \iwahoriheckebasissymbol_w * \iwahoriheckebasissymbol_z $, it is clear that $ \iwahoriheckebasissymbol_w * \iwahoriheckebasissymbol_z $ is supported on $ I w I z I $. Since $ G(F)_1 \triangleleft G(F) $ and $ z $ stabilizes the base alcove $ \alcovesymbol $, $ z $ (meaning its representative in $ N_G(A)(F) $) normalizes $ I $, so $ I w I z I = I w z I $. The integrand of the integral defining $ ( \iwahoriheckebasissymbol_w * \iwahoriheckebasissymbol_z ) ( w z ) $ is supported on $ I w I $. Write $ I w I = w I \cup ( I w I \setminus w I ) $ and notice that the integration over $ w I $ is $ 1 $, since $ dG_I $ is translation-invariant and $ dG_I(I) = 1 $. On the other hand, the support of the integrand $ g \mapsto 1_{IzI} ( g^{-1} w z ) $ is disjoint from the complement $ I w I \setminus w I $, so $ ( \iwahoriheckebasissymbol_w * \iwahoriheckebasissymbol_z ) ( w z ) = 1 $. Next, we prove (\ref{IMPbraid}). Suppose $ w \in W_{\aff} $, $ s \in \Delta_{\aff} $, $ \ell ( s w ) > \ell ( w ) $. As in the previous part, the function $ 1_{ I s I } * 1_{ I w I } $ is supported on $ I s I w I $ and the Tits System Axiom (see \S\ref{SSsemidirectproductdecomps}) implies that $ I s I w I = I s w I $. So $ 1_{ I s I } * 1_{ I w I } $ is supported on $ I s w I $ and it suffices to show that $ ( 1_{ I s I } * 1_{ I w I } ) ( s w ) = 1 $. As before, this is true because the integrand of the integral defining $ ( 1_{ I s I } * 1_{ I w I } ) ( s w ) $ is supported on $ I s I $, is constant $ 1 $ on $ s I $, and is $ 0 $ on the complement $ I s I \setminus s I $. Induction now proves relation (\ref{IMPbraid}). Finally, we prove (\ref{IMPquadratic}). As in the previous parts, examining the integral which defines $ \iwahoriheckebasissymbol_s * \iwahoriheckebasissymbol_s $ shows that the support of $ \iwahoriheckebasissymbol_s * \iwahoriheckebasissymbol_s $ is $ I s I s I = I s I \cup I $. It is clear from the definition of $ \iwahoriheckebasissymbol_s * \iwahoriheckebasissymbol_s $ that the value at the identity $ 1 \in G(F) $, i.e. the coefficient of $ \iwahoriheckebasissymbol_1 $, is $ [ I s I : I ] $ and that the coefficient of $ \iwahoriheckebasissymbol_s $ is the index $ [ ( I s I \cap I s I s ) : I ] $. It suffices now to show that this latter index is simply $ [ I s I : I ] - 1 $. Let $ s, x_1, \ldots, x_k \in s I $ be a system of representatives for the cosets $ I \backslash I s I $. Certainly $ I s I \cap I s I s \subset I x_1 \cup \cdots \cup I x_k $, since $ I $ does not contain representatives for non-identity elements of $ \finiteweylgroup $. On the other hand, $ x_i s \in s I s \subset I \cup I s I $ by choice and the Tits System Axiom (see \S\ref{SSsemidirectproductdecomps}). By choice $ x_i s \notin I $ so $ x_i s \in I s I $ and therefore $ x_i \in I s I s $. Together, $ I s I \cap I s I s = I x_1 \cup \cdots \cup I x_k $, as desired.
\end{proof}

\begin{remark}
A formal consequence of the Iwahori-Matsumoto Presentation is that the parameter system $ \param : \Delta_{\aff} \rightarrow \N $ is conjugation-invariant.
\end{remark}


\subsection{Equality of two related systems of Hecke algebra parameters}

\begin{center}
\emph{The main purpose of this short subsection is to prove that the parameters occurring in various presentations of $ \mathcal{H} ( G; I ) $ are actually the same as those occurring in those of $ \mathcal{H} ( G; \widetilde{I} ) $. See also Remark \ref{Raveragingmap}. This subsection is not logically required for the rest of the paper.}
\end{center}

\begin{lemma} \label{Lgeneralizedtorsionparam}
The inclusions $ \widetilde{I} \subset \widetilde{K} $ and $ I \subset K $ induce a bijection $ \widetilde{I} / I \directedisom \widetilde{K} / K $ and $ [ K : I ] = [ \widetilde{K} : \widetilde{I} ] $
\end{lemma}

\begin{proof}
By Proposition 11.1.4 of \cite{HRo}, every element of $ \widetilde{K} $ is represented modulo $ K $ by an element of $ M(F)^1 $. Since $ M(F)^1 $ is also equal to $ M(F) \cap \widetilde{I} $ (this follows from \S8 of \cite{HRo}, especially Lemma 8.0.1 and Remark 8.0.2), the induced map $ \widetilde{I} / I \rightarrow \widetilde{K} / K $ is surjective. On the other hand, $ \widetilde{I} \cap K \subset G ( F )_1 $ (by Proposition 3 of the Appendix to \cite{PRH}) and fixes the base alcove $ \alcovesymbol $ pointwise, so $ \widetilde{I} \cap K = I $ (Proposition 3 of the Appendix to \cite{PRH} again) and the induced map is also injective. The equality follows since $ [ \widetilde{K} : K ][ K : I ]=[ \widetilde{K} : \widetilde{I} ][ \widetilde{I} : I ] $.
\end{proof}

\begin{remark}
Lemma \ref{Lgeneralizedtorsionparam} appeared earlier in the Appendix of \cite{haines2}.
\end{remark}

By Lemma 5.0.1 of \cite{HRo}, the inclusion $ K \cap N_G(S)(F) \subset N_G(A)(F) $ induces a group isomorphism
\begin{equation} \label{ErepresentKbyfiniteWeyl}
( N_G(S)(F) \cap K ) / T(F)_1 \directedisom \finiteweylgroup
\end{equation}
(the quotient here is denoted by ``$ \widetilde{W}_K^{\sigma} $'' in \cite{HRo}). Unsurprisingly,
\begin{lemma} \label{LdoubleKmodIcosets}
If $ n_w \in N_G(S)(F) \cap K $ denotes a representative of $ w \in \finiteweylgroup $ (which is possible for all $ w \in \finiteweylgroup $ because of (\ref{ErepresentKbyfiniteWeyl})) then the function $ w \mapsto I n_w I $ is a bijection $ \finiteweylgroup \directedisom I \backslash K / I $.
\end{lemma}

\begin{proof}
The function is well-defined because of the careful choice of representatives. First, we prove surjectivity. Let $ k \in K $ be arbitrary. By Proposition 8 / Remark 9 of the Appendix to \cite{PRH}, there is some $ n \in N_G(S)(F) $ for which $ I k I = I n I $ and since $ I \subset K $, necessarily $ n \in N_G(S)(F) \cap K \subset N_G(A)(F) \cap K $. Now, we verify injectivity. Suppose $ n_1, n_2 \in N_G(S)(F) \cap K $ and $ I n_1 I = I n_2 I $. But by the uniqueness statement in Proposition 8 / Remark 9 of the Appendix to \cite{PRH}, $ n_1 n_2^{-1} \in T(F)_1 $, which implies, using (\ref{ErepresentKbyfiniteWeyl}) again, that $ n_1, n_2 $ represent the same element of $ \finiteweylgroup $.
\end{proof}

\begin{prop} \label{Pmacdonaldnumbersareunchanged}
$ [ I w I : I ] = [ \widetilde{I} w \widetilde{I} : \widetilde{I} ] $ for all $ w \in \finiteweylgroup $.
\end{prop}

\begin{proof}
Fix $ w \in \finiteweylgroup $. We first prove that $ I w I / I \canarrow \widetilde{I} w \widetilde{I} / \widetilde{I} $ is injective. Suppose $ \alpha, \beta \in I $ and $ \alpha w \widetilde{I} = \beta w \widetilde{I} $. Then $ w^{-1} ( \beta^{-1} \alpha ) w \in \widetilde{I} $ and so it suffices to show that $ ( ( w^{-1} I w ) \cap \widetilde{I} ) \subset I $. If $ x \in w^{-1} I w $ then by (a corollary of) the Tits System Axiom (see \S\ref{SSsemidirectproductdecomps}), $ x \in I w^{-1} v I $ for some $ v \preceq w $. If $ x \in \widetilde{I} $ also then since $ I \subset \widetilde{I} $, it follows that $ w^{-1} v \in \widetilde{I} $. But $ \widetilde{I} $ does not contain representatives for any non-identity element of $ \finiteweylgroup $, so $ I w^{-1} v I = I $ and $ x \in I $, as desired. Therefore, $ [ I w I : I ] \leq [ \widetilde{I} w \widetilde{I} : \widetilde{I} ] $ for all $ w \in \finiteweylgroup $. By Lemmas \ref{LdoubleKmodIcosets} and \ref{Lgeneralizedtorsionparam}, $ \sum_{w \in \finiteweylgroup} [ I w I : I ] = [ K : I ] = [ \widetilde{K} : \widetilde{I} ] $, and certainly $ [ \widetilde{K} : \widetilde{I} ] \geq \sum_{w \in \finiteweylgroup} [ \widetilde{I} w \widetilde{I} : \widetilde{I} ] $. Combining the two inequalities yields the claim.
\end{proof}

\section{Proof of the Bernstein presentation} \label{Sproofbernsteinrelation}

\subsection{Map to the extended affine Hecke algebra} \label{SSmaptoreducedheckealgebra}

Recall from \S\ref{SSvariouscompatibilities} that $ \Delta_{\aff} $ is also considered to be a Coxeter generating set for $ W_{\aff} ( \overline{\scaledrootsystemsymbol} ) \subset \widetilde{W} ( \Psi_M ) $ and $ \ell $ also as the associated length function on $ \widetilde{W} ( \Psi_M ) $.

Using the based reduced root datum $ ( \Psi_M, \Delta_{\aff} ) $ and the same parameter system $ \param : \Delta_{\aff} \rightarrow \N $ as for $ \mathcal{H} ( G ; I ) $, denote by
\begin{equation*}
\mathcal{H} ( \Psi_M, \Delta_{\aff}, \param )
\end{equation*}
the extended affine Hecke algebra defined in \S3.2 of \cite{lusztig}.

\begin{remark} \label{Rtensorlusztig}
The Hecke algebras used in \cite{lusztig} are actually $ \C[ v, v^{-1} ] $-algebras and use ``exponential'' parameter systems $ L : \widetilde{W} ( \Psi_M ) \rightarrow \N $, i.e. the parameters appearing in the Iwahori-Matsumoto Presentation for $ \mathcal{H} ( \Psi_M, \Delta_{\aff}, \param ) $ (see below) are the indeterminate powers $ v^{2 L(s)} $ instead of the numerical parameters $ \param(s) $ that are used here. This causes no trouble because each numerical parameter $ \param(s) $ is an integer power of $ q $, say $ \param(s) = q^{L(s)} $ (this can be proved by expressing $ I s I / I \cong I /  ( I \cap s I s ) $ as a quotient of two groups of the form $ U^{\natural}_{b, k} $ from \S5.1.16 of \cite{BT2}). Therefore one can define $ \eval : \C[ v, v^{-1} ] \rightarrow \C $ by $ v \mapsto \sqrt{q} $, apply $ - \otimes_{\eval} \C $, and transport any $ \C $-algebra identity from \cite{lusztig} to $ \mathcal{H} ( \Psi_M, \Delta_{\aff}, \param ) $.
\end{remark}

Since $ \otimes $ distributes over arbitrary direct sums, \S3.2 of \cite{lusztig} implies (see Remark \ref{Rtensorlusztig}) that the algebra $ \mathcal{H} ( \Psi_M, \Delta_{\aff}, \param ) $ has a $ \C $-linear basis of elements $ \affineheckebasissymbol_w $ indexed by $ w \in \widetilde{W} ( \Psi_M ) $. By \S3.2 and \S2.1 of \cite{lusztig} the usual Iwahori-Matsumoto relations hold in $ \mathcal{H} ( \Psi_M, \Delta_{\aff}, \param ) $: for all $ s \in \Delta_{\aff} $ and for all $ w, w^{\prime} \in \widetilde{W} ( \Psi_M ) $ satisfying $ \ell ( w w^{\prime} ) = \ell ( w ) + \ell ( w^{\prime} ) $,
\begin{align*}
\affineheckebasissymbol_w \cdot \affineheckebasissymbol_{ w^{\prime} } &= \affineheckebasissymbol_{ w w^{\prime} } \\
\affineheckebasissymbol_s \cdot \affineheckebasissymbol_s &= ( \param(s) - 1 ) \affineheckebasissymbol_s + \param(s) \affineheckebasissymbol_1
\end{align*}

Recall the group homomorphism $ \iota : \widetilde{W} \twoheadrightarrow \widetilde{W} ( \Psi_M ) $ from \S\ref{SSvariouscompatibilities}. Define a $ \C $-linear transformation
\begin{equation*}
\iota_{\mathcal{H}} : \mathcal{H} ( G ; I ) \longrightarrow \mathcal{H} ( \Psi_M, \Delta_{\aff}, \param )
\end{equation*}
by $ \C $-linearly extending the rule $ \iwahoriheckebasissymbol_w \mapsto \affineheckebasissymbol_{ \iota ( w ) } $.

\begin{lemma} \label{Lringhomtoreducedheckealgebra}
$ \iota_{\mathcal{H}} $ is a surjective $ \C $-algebra homomorphism and
\begin{equation*}
\kernel ( \iota_{\mathcal{H}} ) = \{ h = \sum_{w \in \widetilde{W}} h_w \iwahoriheckebasissymbol_w \suchthat \text{ for all } ( u, \tau^{\prime} ) \in \widetilde{W} ( \Psi_M ), \sum_{ \substack{ \tau \in \Omega_G \\ \iota ( \tau ) = \tau^{\prime} } } h_{( u, \tau )} = 0 \}.
\end{equation*}

In particular, if $ \iota_{\mathcal{H}} ( h ) = 0 $ and if $ \Omega_G ( w ) $, the projection of $ w $ onto $ \Omega_G^{\prime} $, is the same for all $ w $ supporting $ h $ then $ h = 0 $.
\end{lemma}

\begin{remark} \label{Raveragingmap}
By Proposition \ref{Pmacdonaldnumbersareunchanged} and isomorphism $ \varphi^{\prime} $ from Lemma \ref{Lcompatibilitysquare}, $ \mathcal{H} ( \Psi_M, \Delta_{\aff}, \param ) = \mathcal{H} ( G ; \widetilde{I} ) $ and $ \iota_{\mathcal{H}} $ is the obvious averaging map, even though averaging maps between Hecke algebras are usually not ring homomorphisms.
\end{remark}

\begin{proof}
It is obvious that $ \iota_{\mathcal{H}} $ is surjective since $ \iota $ itself is. It is clear from Lemma \ref{Laffineweylisomtosyntheticaffineweyl} and the definition of $ \iota_{\mathcal{H}} $ that $ \kernel ( \iota_{\mathcal{H}} ) $ is as described. The final statement is immediate from the description of $ \kernel ( \iota_{\mathcal{H}} ) $. The last thing to verify is that $ \iota_{\mathcal{H}} $ is a ring homomorphism. Since $ \Delta_{\aff} $ is a Coxeter generating set for both $ W_{\aff} $ and $ W_{\aff} ( \overline{\scaledrootsystemsymbol} ) $, suppress $ \iota $ and write $ \iota_{\mathcal{H}} ( \iwahoriheckebasissymbol_s ) = \affineheckebasissymbol_s $ for all $ s \in \Delta_{\aff} $. The Iwahori-Matsumoto Presentation (Proposition \ref{Piwahorimatsumotopresentation}) completely determines the ring structure of $ \mathcal{H} ( G ; I ) $, so it suffices to show \textbf{(1)} that $ \iota_{\mathcal{H}} ( \iwahoriheckebasissymbol_s * \iwahoriheckebasissymbol_s ) = \affineheckebasissymbol_s \cdot \affineheckebasissymbol_s $ for all $ s \in \Delta_{\aff} $, \textbf{(2)} that $ \iota_{\mathcal{H}} ( \iwahoriheckebasissymbol_{s_1} * \cdots * \iwahoriheckebasissymbol_{s_n} ) = \affineheckebasissymbol_{s_1} \cdots \affineheckebasissymbol_{s_n} $ for all $ s_1, \ldots, s_n \in \Delta_{\aff} $ satisfying $ \ell ( s_1 \cdots s_n ) = n $ and \textbf{(3)} that $ \iota_{\mathcal{H}} ( \iwahoriheckebasissymbol_w * \iwahoriheckebasissymbol_{\tau} ) = \affineheckebasissymbol_{\iota ( w ) } \cdot \affineheckebasissymbol_{\iota ( \tau ) } $ for all $ w \in W_{\aff} $ and $ \tau \in \Omega_G^{\prime} $. Statement (1) is true by definition of $ \iota_{\mathcal{H}} $ and the Iwahori-Matsumoto Presentations because the same parameter system $ \param $ is used for both Hecke algebras: $ \iota_{\mathcal{H}} ( \iwahoriheckebasissymbol_s * \iwahoriheckebasissymbol_s ) = \iota_{\mathcal{H}} ( (\param(s)-1) \iwahoriheckebasissymbol_s + \param(s) \iwahoriheckebasissymbol_1 ) \defeq (\param(s)-1) \iota_{\mathcal{H}} ( \iwahoriheckebasissymbol_s ) + \param(s) \iota_{\mathcal{H}} ( \iwahoriheckebasissymbol_1 ) = (\param(s)-1) \affineheckebasissymbol_s + \param(s) \affineheckebasissymbol_1 = \affineheckebasissymbol_s \cdot \affineheckebasissymbol_s $. Similarly, since the same length function $ \ell $ is used for the underlying Coxeter group of both Hecke algebras, if $ s_1, \ldots, s_n \in \Delta_{\aff} $ and $ \ell ( s_1 \cdots s_n ) = n $ then the Iwahori-Matsumoto Presentations for both Hecke algebras yields $ \iota_{\mathcal{H}} ( \iwahoriheckebasissymbol_{s_1} * \cdots * \iwahoriheckebasissymbol_{s_n} ) = \iota_{\mathcal{H}} ( \iwahoriheckebasissymbol_{ s_1 \cdots s_n } ) \defeq \affineheckebasissymbol_{ \iota ( s_1 \cdots s_n ) } = \affineheckebasissymbol_{ \iota ( s_1 ) \cdots \iota ( s_n ) } = \affineheckebasissymbol_{ s_1 } \cdots \affineheckebasissymbol_{ s_n } $. Finally, it is clear from Lemma \ref{Laffineweylisomtosyntheticaffineweyl} and the Iwahori-Matsumoto Presentations that $ \iota_{\mathcal{H}} ( \iwahoriheckebasissymbol_w * \iwahoriheckebasissymbol_{\tau} ) = \iota_{\mathcal{H}} ( \iwahoriheckebasissymbol_{w \tau} ) \defeq \affineheckebasissymbol_{\iota ( w \tau ) } = \affineheckebasissymbol_{\iota ( w ) \iota ( \tau ) } = \affineheckebasissymbol_{\iota ( w ) } \cdot \affineheckebasissymbol_{\iota ( \tau ) } $.
\end{proof}

\subsection{Dominance in $ \Omega_M $} \label{SSdominance}

Recall from \S\ref{SSgenuinerootdatum} the $ \R $-linear isomorphism $ \cochargroup ( A ) \otimes_{\Z} \R \directedisom \overline{\Omega}_M \otimes_{\Z} \R $. It is clear from the construction of $ \Psi_M $ (see \S\ref{SSgenuinerootdatum}) that this isomorphism preserves the simplicial structures coming from $ \scaledrootsystemsymbol $ resp. $ \overline{\scaledrootsystemsymbol} $, so abuse notation and identify $ \apartmentsymbol = \overline{\Omega}_M \otimes_{\Z} \R $ also. Accordingly, consider $ \alcovesymbol \subset \mathcal{C} \subset \overline{\Omega}_M \otimes_{\Z} \R $.

Recall the standard definition of dominance for extended affine Weyl groups: $ \lambda \in \overline{\Omega}_M $ is \emph{dominant} iff $ t_{\lambda} ( \alcovesymbol ) \subset \mathcal{C} $. Equivalently, $ \lambda $ is dominant iff $ t_{\lambda} ( \vertexsymbol ) \in \overline{\mathcal{C}} $ or, in coordinates, $ \lambda $ is dominant iff $ \langle \alpha, \lambda \rangle_M \geq 0 $ for all $ \alpha \in \Delta_{\finiteweylsymbol} $.

Because of Lemma \ref{Lcompatibilitysquare}, the action of $ \Omega_M \subset \widetilde{W} $ by translations of $ \apartmentsymbol $ agrees with the canonical action of $ \overline{\Omega}_M \subset \widetilde{W} ( \Psi_M ) $ by translations of $ \apartmentsymbol = \overline{\Omega}_M \otimes_{\Z} \R $. This suggests that the same definition should be made for $ \Omega_M $:

\begin{defn}
$ \lambda \in \Omega_M $ is \emph{dominant} iff $ \iota ( \lambda ) \in \overline{\Omega}_M $ is dominant. The subset of all dominant elements is denoted by $ \Omega_M^{\dom} $.
\end{defn}
Note in particular that $ ( \Omega_M )_{\tor} \subset \Omega_M^{\dom} $ by Lemma \ref{LMtorsioninsideGtorsion}.

\begin{lemma} \label{Ldominancefacts}
The following expected facts regarding dominance in $ \Omega_M $ are true:
\begin{enumerate}
\item \label{Ldomrotate} If $ \mu \in \Omega_M $ then there exists $ w \in \finiteweylgroup $ such that $ w ( \mu ) \in  \Omega_M^{\dom} $.

\item \label{Ldomshift} If $ \mu_1, \ldots, \mu_r \in \Omega_M $ then there exists $ \lambda \in \Omega_M^{\dom} $ such that all $ \lambda + \mu_i \in \Omega_M^{\dom} $.

\item \label{Ldomconcat} If $ \lambda, \mu \in \Omega_M^{\dom} $ then $ \ell ( t_{\lambda} \circ t_{\mu} ) = \ell ( t_{\lambda} ) + \ell ( t_{\mu} ) $.

\item \label{Ldomsplit} If $ \lambda \in \Omega_M^{\dom} $ and $ w \in \finiteweylgroup $ then $ \ell ( w \circ t_{\lambda} ) = \ell ( w ) + \ell ( t_{\lambda} ) $.

\item \label{Ldomorbit} If $ \lambda \in \Omega_M $ then $ \ell ( t_{w ( \lambda )} ) = \ell ( t_{\lambda} ) $ for all $ w \in \finiteweylgroup $.
\end{enumerate}
\end{lemma}

\begin{proof}
Because of Lemma \ref{Lcompatibilitysquare}, the corollary to Lemma \ref{Laffineweylisomtosyntheticaffineweyl}, and the definition of dominance, it suffices to know these properties for the extended affine Weyl group $ \widetilde{W} ( \Psi_M ) $. These corresponding facts for extended affine Weyl groups are all very well-known, but we make a few additional comments since there seems to be no single convenient reference for all of them. Fact (\ref{Ldomrotate}) is true simply because finite Weyl groups fix the origin $ \vertexsymbol $ and are transitive on Weyl chambers: the transformation of $ \apartmentsymbol $ by $ w ( \mu ) $ is $ w \circ t_{\lambda} \circ w^{-1} $ and transitivity means that $ w \in \finiteweylgroup $ can be chosen so that $ w ( t_{\lambda} ( w^{-1} ( \vertexsymbol ) ) ) \in \overline{\mathcal{C}} $. Facts (\ref{Ldomshift}), (\ref{Ldomconcat}), and (\ref{Ldomsplit}) are true because the analogous facts for $ \widetilde{W} ( \Psi_M ) $ are known (see the proof of Lemma 3.4, \S1.4(g), and \S1.4(f) in \cite{lusztig}). Fact (\ref{Ldomorbit}) also follows from the analogous known fact for $ \widetilde{W} ( \Psi_M ) $, which itself follows immediately from the formula that defines length (see \S1.4(a) of \cite{lusztig}) on an extended affine Weyl group (together with the elementary fact that applying an element of a finite Weyl group to a positive system is the same as multiplying certain of those positive roots by $ -1 $).
\end{proof}

\subsection{Definition and basic properties of the $ \iwahoriheckethetaelementsymbol $-elements}

Recall that for $ w \in \widetilde{W} ( \Psi_M ) $ it is customary to define
\begin{equation*}
\normalizedaffineheckebasissymbol_w \defeq \param ( w )^{-\frac{1}{2}} \affineheckebasissymbol_w
\end{equation*}
and that the $ \theta $-elements that occur in the Bernstein Presentation for $ \mathcal{H} ( \Psi_M, \Delta_{\aff}, \param ) $ are defined as follows: for $ \lambda \in \overline{\Omega}_M $, there exist dominant $ \lambda_1, \lambda_2 \in \overline{\Omega}_M^{\dom} $ such that $ \lambda = \lambda_1 - \lambda_2 $ and
\begin{equation*}
\affineheckethetaelementsymbol_{\lambda} \defeq \normalizedaffineheckebasissymbol_{\lambda_1} \cdot \normalizedaffineheckebasissymbol_{\lambda_2}^{-1}.
\end{equation*}

Similarly,
\begin{defn}
For any $ w \in \widetilde{W} $, define
\begin{equation*}
\normalizediwahoriheckebasissymbol_w \defeq \param ( w )^{-\frac{1}{2}} \iwahoriheckebasissymbol_w
\end{equation*}
For any $ \lambda \in \Omega_M $, let $ \lambda_1, \lambda_2 \in \Omega_M^{\dom} $ be such that $ \lambda = \lambda_1 - \lambda_2 $ (Lemma \ref{Ldominancefacts}(\ref{Ldomshift})) and define
\begin{equation*}
\iwahoriheckethetaelementsymbol_{\lambda} \defeq \normalizediwahoriheckebasissymbol_{\lambda_1} * \normalizediwahoriheckebasissymbol_{\lambda_2}^{-1}
\end{equation*}
Well-definedness follows from Lemma \ref{Ldominancefacts}(\ref{Ldomconcat}) and the fact that $ \Omega_M $ is abelian.
\end{defn}

The following lemma will be used frequently throughout the rest of the paper to transport results from \cite{lusztig} to $ \mathcal{H} ( G ; I ) $:
\begin{lemma} \label{Lthetaelementsarecompatible}
$ \iota_{\mathcal{H}} ( \iwahoriheckethetaelementsymbol_{\lambda} ) = \affineheckethetaelementsymbol_{\iota (\lambda)} $ in $ \mathcal{H} ( \Psi_M, \Delta_{\aff}, \param ) $ for all $ \lambda \in \Omega_M $.
\end{lemma}

\begin{proof}
This is obvious since $ \iota_{\mathcal{H}} ( \iwahoriheckethetaelementsymbol_{\lambda} ) = \normalizedaffineheckebasissymbol_{ \iota ( \lambda_1 ) } \cdot \normalizedaffineheckebasissymbol_{ \iota ( \lambda_2 ) }^{-1} $ by Lemma \ref{Lringhomtoreducedheckealgebra} and since $ \iota \vert_{\Omega_M} : \Omega_M \canarrow \overline{\Omega}_M $ preserves dominance by definition.
\end{proof}

Recall that $ \Omega_G ( w ) $ denotes the projection of $ w \in \widetilde{W} $ onto $ \Omega_G^{\prime} $. The following lemma will be used frequently throughout the rest of the paper in combination with the description of $ \kernel ( \iota_{\mathcal{H}} ) $ in Lemma \ref{Lringhomtoreducedheckealgebra}:

\begin{lemma} \label{Lomegacomponentsofthetas}
If $ \mu \in \Omega_M $ and $ s \in \Delta_{\aff} $ then $ \Omega_G ( w ) = \Omega_G ( t_{\mu} ) $ for all $ w \in \widetilde{W} $ supporting $ \iwahoriheckethetaelementsymbol_{\mu} $, or supporting $ \iwahoriheckethetaelementsymbol_{\mu} * \iwahoriheckebasissymbol_s $, or supporting $ \iwahoriheckebasissymbol_s * \iwahoriheckethetaelementsymbol_{\mu} - \iwahoriheckethetaelementsymbol_{s(\mu)} * \iwahoriheckebasissymbol_s $.
\end{lemma}

\begin{proof}
This is immediate from the Iwahori-Matsumoto Presentation and the definition of the $ \Theta $-elements.
\end{proof}

The following fact, Corollary 5.7(1) of \cite{haines1}, will be used in \S\ref{SSbernfunctionsarebasis}:
\begin{lemma} \label{Lhainesfact}
Fix $ \lambda \in \Omega_M $. If $ w $ supports $ \iwahoriheckethetaelementsymbol_{\lambda} $ then $ w \preceq \lambda $.
\end{lemma}

\begin{proof}
Superficially, the result in \cite{haines1} applies only to extended affine Hecke algebras on reduced root data with constant parameter systems, but the same proof applies without change to $ \mathcal{H} ( G ; I ) $, as we now explain. By Iwahori-Matsumoto relation (\ref{IMPnormalizer}) and the definition of the $ \Theta $-elements (and how the Bruhat-Chevalley order is extended from $ W_{\aff} $), it suffices to show that if $ u, v \in W_{\aff} $ then $ w \preceq u v^{-1} $ for all $ w $ supporting $ \iwahoriheckebasissymbol_u * \iwahoriheckebasissymbol_{v}^{-1} $ (we are merely repeating an observation from \cite{haines1} here). But the $ \C $-subspace of $ \mathcal{H} ( G ; I ) $ spanned by $ \{ \iwahoriheckebasissymbol_w \suchthat w \in W_{\aff} \} $ is a $ \C $-subalgebra and clearly isomorphic to the affine Hecke algebra $ \mathcal{H} ( W_{\aff} ( \scaledrootsystemsymbol ), \Delta_{\aff}, \param ) $. Therefore, it suffices to see why Corollary 5.7(1) of \cite{haines1} is true using an arbitrary parameter systems. But this is easy to see since the required first half of the proof of Proposition 5.4 in \cite{haines1} works verbatim for an arbitrary choice of parameters and Proposition 5.5 in \cite{haines1} is purely a statement about Coxeter groups.
\end{proof}

\begin{corollary}
Fix $ \mathcal{O} \in \Omega_M / \finiteweylgroup $. If $ \lambda, \mu \in \mathcal{O} $ and $ \lambda \neq \mu $ then $ \iwahoriheckethetaelementsymbol_{\mu} ( \lambda ) = 0 $.
\end{corollary}

\begin{proof}
If $ \iwahoriheckethetaelementsymbol_{\mu} ( \lambda ) \neq 0 $ then $ \lambda \preceq \mu $ by Lemma \ref{Lhainesfact}. But since $ \lambda $ and $ \mu $ are in the same $ \finiteweylgroup $-orbit, $ \ell ( \lambda ) = \ell ( \mu ) $ (Lemma \ref{Ldominancefacts}(\ref{Ldomorbit})). Together, this means that $ \lambda $ and $ \mu $ have the same $ \Omega_G $-component and the same $ W_{\aff} $-component, i.e. $ \lambda = \mu $.
\end{proof}

\subsection{Proof of the Bernstein relations} \label{SScommutationrelation}

\begin{prop} \label{Pbernsteinadditivityrelation}
For all $ \lambda, \mu \in \Omega_M $,
\begin{equation*}
\iwahoriheckethetaelementsymbol_{\lambda} * \iwahoriheckethetaelementsymbol_{\mu} = \iwahoriheckethetaelementsymbol_{\lambda + \mu}
\end{equation*}
\end{prop}

\begin{proof}
Suppose $ \lambda_1, \lambda_2, \mu_1, \mu_2 \in \Omega_M^{\dom} $ are such that $ \lambda = \lambda_1 - \lambda_2 $ and $ \mu = \mu_1 - \mu_2 $. Since $ \ell ( t_{\lambda} \circ t_{\mu} ) = \ell ( \lambda ) + \ell ( \mu ) $ for all $ \lambda, \mu \in \Omega_M^{\dom} $ (Lemma \ref{Ldominancefacts}(\ref{Ldomconcat})) and since $ \Omega_M $ is abelian, $ \iwahoriheckebasissymbol_{\lambda_2}^{-1} * \iwahoriheckebasissymbol_{\mu_1} = \iwahoriheckebasissymbol_{\mu_1} * \iwahoriheckebasissymbol_{\lambda_2}^{-1} $. The additivity of lengths also implies that $ \iwahoriheckebasissymbol_{\lambda_i} * \iwahoriheckebasissymbol_{\mu_i} = \iwahoriheckebasissymbol_{\lambda_i + \mu_i} $ and that $ \param ( \lambda_i ) \param ( \mu_i ) = \param ( \lambda_i + \mu_i ) $. Combining all this proves the claim since $ \lambda_i + \mu_i \in \Omega_M^{\dom} $ and $ \lambda + \mu = ( \lambda_1 + \mu_1 ) - ( \lambda_2 + \mu_2 ) $.
\end{proof}

Denote by $ s \mapsto \widetilde{s} $ the involution on $ \Delta_{\aff} $ (only non-trivial for reflections coming from type $ C $ subsystems) described in \S2.4 of \cite{lusztig} and define
\begin{align*}
\indexedparam{0}{s} &\defeq \param ( s ) - 1 \\
\indexedparam{1}{s} &\defeq \param ( s )^{\frac{1}{2}} \param ( \widetilde{s} )^{\frac{1}{2}} - \param ( s )^{\frac{1}{2}} \param ( \widetilde{s} )^{-\frac{1}{2}}
\end{align*}
For any $ \jmath \in \Z $, denote by $ \overline{\jmath} \in \{ 0, 1 \} $ its remainder mod $ 2 $. Note that if $ \param ( s ) = \param ( \widetilde{s} ) $ then $ \indexedparam{0}{s} = \indexedparam{1}{s} $.

Recall from Proposition 3.6 of \cite{lusztig} that there is a commutation relation for the extended affine Hecke algebra $ \mathcal{H} ( \Psi_M, \Delta_{\aff}, \param ) $: if $ x \in \overline{\Omega}_M $ and $ s \in \Delta_{\finiteweylsymbol} $ ($ s = s_{\alpha} $, $ \alpha \in \overline{\scaledrootsystemsymbol} $) then
\begin{equation} \label{Eoldbernsteincommutationrelation}
\affineheckebasissymbol_s \cdot \affineheckethetaelementsymbol_{x} - \affineheckethetaelementsymbol_{s(x)} \cdot \affineheckebasissymbol_s = \sum_{\jmath=0}^{\langle \alpha, x \rangle_M - 1} \indexedparam{\overline{\jmath}}{s} \affineheckethetaelementsymbol_{ x - \jmath \alpha^{\vee} }
\end{equation}
(note that the elements $ x - \jmath \alpha^{\vee} $ make sense because, from \S\ref{SSgenuinerootdatum}, $ \overline{\scaledrootsystemsymbol}^{\vee} \subset \overline{\Omega}_M $). To expand the relation appearing in \cite{lusztig} to the form (\ref{Eoldbernsteincommutationrelation}) that is more convenient here, use the formal identities
\begin{align*}
\affineheckethetaelementsymbol_x - \affineheckethetaelementsymbol_{s(x)} &= \affineheckethetaelementsymbol_x \cdot ( 1 - \affineheckethetaelementsymbol_{ - \langle \alpha, x \rangle \alpha^{\vee} } ) \\
1 - \affineheckethetaelementsymbol_{- k n \alpha^{\vee}} &= ( 1 - \affineheckethetaelementsymbol_{- k \alpha^{\vee} } ) \cdot ( 1 + \affineheckethetaelementsymbol_{ - k \alpha^{\vee} } + \cdots + \affineheckethetaelementsymbol_{- k ( n - 1 ) \alpha^{\vee} } )
\end{align*}
that result from the additivity relation (Proposition \ref{Pbernsteinadditivityrelation}) above ($ k = 1, 2 $ and $ n $ should be such that $ k n = \langle \alpha, x \rangle_M $ and $ k = 2 $ iff $ \alpha \in 2 \overline{\Omega}_M^{\vee} $).


\begin{remark}
Relation (\ref{Eoldbernsteincommutationrelation}) is not literally Proposition 3.6 of \cite{lusztig}, but instead is the image of that relation under specialization $ v \mapsto \sqrt{q} $. See Remark \ref{Rtensorlusztig}.
\end{remark}

\begin{prop} \label{Pbernsteincommutationrelation}
Fix $ s \in \Delta_{\finiteweylsymbol} $ and $ \lambda \in \Omega_M $. Let $ \alpha \in \scaledrootsystemsymbol $ be such that $ s = s_{\alpha} $ and set $ N \defeq \langle \alpha, \iota ( \lambda ) \rangle_M \defeq \langle \alpha, \lambda \rangle_{\R} $. Then
\begin{equation} \label{Enewbernsteincommutationrelation}
\iwahoriheckebasissymbol_s * \iwahoriheckethetaelementsymbol_{\lambda} - \iwahoriheckethetaelementsymbol_{s(\lambda)} * \iwahoriheckebasissymbol_s = \sum_{\jmath=0}^{N-1} \indexedparam{\overline{\jmath}}{s} \iwahoriheckethetaelementsymbol_{ \lambda - \jmath \alpha^{\vee} }
\end{equation}
\end{prop}

\begin{remark}
Recall that if $ \alpha \in \scaledrootsystemsymbol $ then $ \alpha^{\vee} $ may be merely a \emph{rational} multiple of some ``true'' coroot in $ \relativerootsystem^{\vee} $.
\end{remark}

\begin{proof}
Define $ \mathbf{\Theta} $ to be the right-hand-side of (\ref{Enewbernsteincommutationrelation}). By Lemmas \ref{Lthetaelementsarecompatible} and \ref{Lringhomtoreducedheckealgebra}, $ \iota_{\mathcal{H}} ( \iwahoriheckebasissymbol_s * \iwahoriheckethetaelementsymbol_{\lambda} - \iwahoriheckethetaelementsymbol_{s(\lambda)} * \iwahoriheckebasissymbol_s ) $ and $ \iota_{\mathcal{H}} ( \mathbf{\Theta} ) $ are the left-hand-side and right-hand-side of (\ref{Eoldbernsteincommutationrelation}) for $ x \defeq \iota ( \lambda ) $, so $ \iwahoriheckebasissymbol_s * \iwahoriheckethetaelementsymbol_{\lambda} - \iwahoriheckethetaelementsymbol_{s(\lambda)} * \iwahoriheckebasissymbol_s - \mathbf{\Theta} \in \kernel ( \iota_{\mathcal{H}} ) $. By Lemma \ref{Lringhomtoreducedheckealgebra}, the desired relation (\ref{Enewbernsteincommutationrelation}) will follow immediately if it is verified that $ \Omega_G ( w ) $ is the same for all $ w $ supporting $ \iwahoriheckebasissymbol_s * \iwahoriheckethetaelementsymbol_{\lambda} - \iwahoriheckethetaelementsymbol_{s(\lambda)} * \iwahoriheckebasissymbol_s - \mathbf{\Theta} $, which we now do. Since translations by the coroot system $ \scaledrootsystemsymbol^{\vee} $ are elements of $ W_{\aff} $, it is clear that $ \Omega_G ( \lambda - \jmath \alpha^{\vee} ) $ is the same for all $ \jmath $, so Lemma \ref{Lomegacomponentsofthetas} implies that $ \Omega_G ( w ) = \Omega_G ( \lambda ) $ for all $ w $ supporting $ \mathbf{\Theta} $. By Lemma \ref{Lomegacomponentsofthetas} again, $ \Omega_G ( w ) = \Omega_G ( \lambda ) $ for all $ w $ supporting $ \iwahoriheckebasissymbol_s * \iwahoriheckethetaelementsymbol_{\lambda} - \iwahoriheckethetaelementsymbol_{s(\lambda)} * \iwahoriheckebasissymbol_s $ also.
\end{proof}

\subsection{The Bernstein products are a basis}

\begin{prop} \label{Pnoncentralbernsteinbasis}
The collection $ \{ \iwahoriheckethetaelementsymbol_{\lambda} * \iwahoriheckebasissymbol_w \suchthat \lambda \in \Omega_M, w \in \finiteweylgroup \} $ is a $ \C $-linear basis for $ \mathcal{H} ( G ; I ) $.
\end{prop}

\begin{proof}
We first prove linear-independence. Suppose $ ( \lambda_1, w_1 ), \ldots, ( \lambda_r, w_r ) \in \Omega_M \times \finiteweylgroup $ (distinct pairs) and $ c_1, \ldots, c_r \in \C $ are such that $ c_1 \iwahoriheckethetaelementsymbol_{\lambda_1} * \iwahoriheckebasissymbol_{w_1} + \cdots + c_r \iwahoriheckethetaelementsymbol_{\lambda_r} * \iwahoriheckebasissymbol_{w_r} = 0 $ in $ \mathcal{H} ( G ; I ) $. For each $ i $, Lemma \ref{Lomegacomponentsofthetas} says that $ \Omega_G ( w ) = \Omega_G ( \lambda_i ) $ for all $ w $ supporting $ \iwahoriheckethetaelementsymbol_{\lambda_i} * \iwahoriheckebasissymbol_{w_i} $. Therefore, we may assume that $ \Omega_G ( \lambda_i ) $ is the same for all $ i $ (otherwise the left-hand-side of the dependence relation would be a sum of multiple functions with \emph{disjoint} supports). Apply $ \iota_{\mathcal{H}} $ to the dependence relation and use Lemmas \ref{Lringhomtoreducedheckealgebra} and \ref{Lthetaelementsarecompatible} to get $ c_1 \affineheckethetaelementsymbol_{ \iota ( \lambda_1 ) } \cdot \affineheckebasissymbol_{w_1} + \cdots + c_r \affineheckethetaelementsymbol_{ \iota ( \lambda_r ) } \cdot \affineheckebasissymbol_{w_r} = 0 $. Because $ \otimes $ distributes over arbitrary direct sums, Proposition 3.7 of \cite{lusztig} implies (see Remark \ref{Rtensorlusztig}) that $ \left\{ \affineheckethetaelementsymbol_{ \mu } \cdot \affineheckebasissymbol_w \suchthat \mu \in \overline{\Omega}_M, w \in \finiteweylgroup \right\} $ is a $ \C $-linear basis for $ \mathcal{H} ( \Psi_M, \Delta_{\aff}, \param ) $. This would imply that $ c_i = 0 $ for all $ i $ as desired if it were true that the pairs $ ( \iota ( \lambda_i ), w_i ) \in \widetilde{W} ( \Psi_M ) $ are all distinct, which we now prove. Fix $ i \neq j $. If $ w_i \neq w_j $ then the claim is true, so assume that $ w_i = w_j $. If $ \iota ( \lambda_i ) = \iota ( \lambda_j ) $ then $ \lambda_i - \lambda_j \in ( \Omega_M )_{\tor} $ by Lemma \ref{Lcompatibilitysquare}. But $ ( \Omega_M )_{\tor} \subset ( \Omega_G^{\prime} )_{\tor} $ by Lemma \ref{LMtorsioninsideGtorsion} so by the assumption on $ \Omega_G ( \lambda_i ) $ made at the beginning of this proof, $ \iota ( \lambda_i ) = \iota ( \lambda_j ) $ implies $ \lambda_i = \lambda_j $, which contradicts the assumed distinctness $ ( \lambda_i, w_i ) \neq ( \lambda_j, w_j ) $. Now we prove spanning. Elements of $ \widetilde{W} $ will sometimes be denoted by pairs using the semidirect product $ \widetilde{W} = \Omega_M \rtimes \finiteweylgroup $. If $ u \in \finiteweylgroup $ and $ \lambda \in \Omega_M^{\dom} $ then $ \ell ( u \circ t_{\lambda} ) = \ell ( t_{\lambda} ) + \ell ( u ) $ (Lemma \ref{Ldominancefacts}(\ref{Ldomsplit})) and so by Iwahori-Matsumoto relation (\ref{IMPbraid}), $ \iwahoriheckebasissymbol_{(0, u)} * \iwahoriheckebasissymbol_{(\lambda, 1)} = \iwahoriheckebasissymbol_{ ( u ( \lambda ), u ) } $. So, for any $ \mu \in \Omega_M $ there exists $ u \in \finiteweylgroup $ and $ \lambda \in \Omega_M^{\dom} $ such that $ \iwahoriheckebasissymbol_{(0, u)} * \iwahoriheckebasissymbol_{(\lambda, 1)} = \iwahoriheckebasissymbol_{(\mu, u)} $ (Lemma \ref{Ldominancefacts}(\ref{Ldomrotate})). By the Iwahori-Matsumoto Presentation, for any $ v \in \finiteweylgroup $ the product $ \iwahoriheckebasissymbol_{(\mu, u)} * \iwahoriheckebasissymbol_{(0, v)} $ is the sum of a (non-zero) multiple of $ \iwahoriheckebasissymbol_{(\mu, uv)} $ and a linear combination of $ \iwahoriheckebasissymbol_{(\mu, u v^{\prime} )} $ for various $ v^{\prime} \prec v $. By induction on length, it is clear that the span of the set of triple-products $ \iwahoriheckebasissymbol_u * \iwahoriheckebasissymbol_{\lambda} * \iwahoriheckebasissymbol_v $ ($ u, v \in \finiteweylgroup $, $ \lambda \in \Omega_M^{\dom} $) contains $ \iwahoriheckebasissymbol_{w} $ for all $ w \in \widetilde{W} $ and therefore is all of $ \mathcal{H} ( G ; I ) $. But $ \normalizediwahoriheckebasissymbol_{\lambda} = \iwahoriheckethetaelementsymbol_{\lambda} $ if $ \lambda \in \Omega_M^{\dom} $, so by the commutation relation (Proposition \ref{Pbernsteincommutationrelation}), each such $ \iwahoriheckebasissymbol_u * \iwahoriheckebasissymbol_{\lambda} $ can be replaced by a $ \C $-linear combination of elements of the form $ \iwahoriheckethetaelementsymbol_{x} * \iwahoriheckebasissymbol_w $.
\end{proof}

\section{The center of the Iwahori-Hecke algebra} \label{Scenter}

\subsection{Definition of the central Bernstein functions}

Recall that $ \Omega_M / \finiteweylgroup $ denotes the set of $ \finiteweylgroup $-conjugacy classes in $ \widetilde{W} $ of translations by elements of $ \Omega_M $.

\begin{defn}
For each class $ \mathcal{O} \in \Omega_M / \finiteweylgroup $, define
\begin{equation*}
z_{ \mathcal{O} } \defeq \sum_{\lambda \in \mathcal{O}} \iwahoriheckethetaelementsymbol_{\lambda} \in \mathcal{H} ( G ; I )
\end{equation*}
\end{defn}

\begin{lemma} \label{Lcentralbernelementsarecompatible}
$ \iota ( \mathcal{O} ) $ is a $ \finiteweylgroup $-orbit in $ \overline{\Omega}_M $ and $ \iota_{\mathcal{H}} ( z_{ \mathcal{O} } ) $ is equal to the central Bernstein function $ z_{\iota ( \mathcal{O} )} \in \mathcal{H} ( \Psi_M, \Delta_{\aff}, \param ) $ defined in Proposition 3.11 of \cite{lusztig}. In particular, $ \iota_{\mathcal{H}} ( z_{ \mathcal{O} } ) \in Z [ \mathcal{H} ( \Psi_M, \Delta_{\aff}, \param ) ] $.
\end{lemma}

\begin{proof}
The first claim is obvious from Lemma \ref{Lcompatibilitysquare}. By Lemma \ref{Lthetaelementsarecompatible}, $ \iota_{\mathcal{H}} ( z_{ \mathcal{O} } ) = \sum_{\lambda \in \mathcal{O}} \affineheckethetaelementsymbol_{\iota ( \lambda )} $. As in the proof of Proposition \ref{Pnoncentralbernsteinbasis}, the facts that $ \Omega_G ( \lambda ) = \Omega_G ( \lambda^{\prime} ) $ for all $ \lambda, \lambda^{\prime} \in \mathcal{O} $ (because $ \Omega_G $ is abelian) and $ ( \Omega_M )_{\tor} \subset ( \Omega_G^{\prime} )_{\tor} $ (because of Lemma \ref{LMtorsioninsideGtorsion}) imply that $ \iota ( \lambda ) \neq \iota ( \lambda^{\prime} ) $ for all $ \lambda \neq \lambda^{\prime} $ in $ \mathcal{O} $. It follows that $ \# \iota ( \mathcal{O} ) = \# \mathcal{O} $ and therefore $ \sum_{\lambda \in \mathcal{O}} \affineheckethetaelementsymbol_{\iota ( \lambda )} = \sum_{\mu \in \iota ( \mathcal{O} )} \affineheckethetaelementsymbol_{ \mu } \defeq z_{\iota ( \mathcal{O} )} $.
\end{proof}

\begin{prop} \label{Pweylorbitfunctionsarecentral}
$ z_{ \mathcal{O} } \in Z [ \mathcal{H} ( G ; I ) ] $ for all $ \mathcal{O} \in \Omega_M / \finiteweylgroup $.
\end{prop}

\begin{proof}
Fix $ \mathcal{O} \in \Omega_M / \finiteweylgroup $. By the additivity relation (Proposition \ref{Pbernsteinadditivityrelation}), $ \iwahoriheckethetaelementsymbol_{\mu} $ commutes with $ z_{ \mathcal{O} } $ for all $ \mu \in \Omega_M $, so by Proposition \ref{Pnoncentralbernsteinbasis} it suffices to show that $ z_{ \mathcal{O} } * \iwahoriheckebasissymbol_s = \iwahoriheckebasissymbol_s * z_{ \mathcal{O} } $ for all $ s \in \Delta_{\finiteweylsymbol} $. Fix $ s \in \Delta_{\finiteweylsymbol} $. By Lemma \ref{Lcentralbernelementsarecompatible}, $ \iota_{\mathcal{H}} ( z_{ \mathcal{O} } * \iwahoriheckebasissymbol_s - \iwahoriheckebasissymbol_s * z_{ \mathcal{O} } ) = 0 $. By Lemma \ref{Lringhomtoreducedheckealgebra}, the desired commutativity will follow immediately if we verify that $ \Omega_G ( w ) $ is the same for all $ w $ supporting $ z_{ \mathcal{O} } * \iwahoriheckebasissymbol_s - \iwahoriheckebasissymbol_s * z_{ \mathcal{O} } $, which we do next. Let $ \alpha \in \scaledrootsystemsymbol $ be the root associated to $ s $. Since $ s ( \mathcal{O} ) = \mathcal{O} $, the commutation relation (Proposition \ref{Pbernsteincommutationrelation}) implies that $ z_{ \mathcal{O} } * \iwahoriheckebasissymbol_s - \iwahoriheckebasissymbol_s * z_{ \mathcal{O} } $ is a linear combination of elements $ \iwahoriheckethetaelementsymbol_{ \lambda - \jmath \alpha^{\vee} } $ for various $ \lambda \in \mathcal{O} $ and $ \jmath $. As noted in the above proof of the commutation relation, $ \Omega_G ( \lambda - \jmath \alpha^{\vee} ) = \Omega_G ( \lambda ) $ for all $ \lambda \in \mathcal{O} $ and all $ \jmath $. Since $ \Omega_G ( \lambda ) = \Omega_G ( \lambda^{\prime} ) $ for all $ \lambda, \lambda^{\prime} \in \mathcal{O} $ ($ \Omega_G $ is abelian), Lemma \ref{Lomegacomponentsofthetas} implies that $ \Omega_G ( w ) $ is the same for all $ w $ supporting $ z_{ \mathcal{O} } * \iwahoriheckebasissymbol_s - \iwahoriheckebasissymbol_s * z_{ \mathcal{O} } $.
\end{proof}

\subsection{Compatibility with \cite{roro}} \label{SSrorocompat}

\begin{center}
\emph{In the next subsection \S\ref{SSbernfunctionsarebasis}, the main theorem from \cite{roro} is used to prove that $ \{ z_{ \mathcal{O} } \}_{ \mathcal{O} \in \Omega_M / \finiteweylgroup } $ spans $ Z [ \mathcal{H} ( G ; I ) ] $. In \cite{roro}, it was assumed for convenience that the root system defining the affine Weyl group was irreducible, and this subsection shows that this assumption is not really necessary to the conclusion.}
\end{center}

Denote by $ \myaut ( \apartmentsymbol_i^{\semisimple} ) $ the group of invertible affine transformations of $ \apartmentsymbol_i^{\semisimple} $ (see \S\ref{SSirredsubsystems}). For each $ i $, restriction defines a group homomorphism $ \widetilde{W} \rightarrow \myaut ( \apartmentsymbol_i^{\semisimple} ) $, and the image of $ \widetilde{W} $ under this group homomorphism is denoted $ \widetilde{W}^i $. Similarly, denote by $ W_{\aff}^i $, $ \Omega^i $, $ \Lambda^i $, $ \finiteweylgroup^i $ the images of the subgroups $ W_{\aff} $, $ \Omega_G $, $ \Omega_M $, $ \finiteweylgroup $.

Because of the known semidirect product decompositions from \S\ref{SSsemidirectproductdecomps}, the image $ \widetilde{W}^i \subset \myaut ( \apartmentsymbol_i^{\semisimple} ) $ is equal to the product of subgroups $ W_{\aff}^i \cdot \Omega^i $ and also to the product of subgroups $ \Lambda^i \cdot \finiteweylgroup^i $. It is immediate from the definition that $ W_{\aff}^i = W_{\aff} ( \scaledrootsystemsymbol_i ) $ and $ \finiteweylgroup^i = \finiteweylgroup ( \scaledrootsystemsymbol_i ) $, that $ \Omega^i $ stabilizes $ \scaledrootsystemsymbol_i $ (since $ \Omega_G $ already does) and that $ \Lambda_i $ consists of translations of $ \apartmentsymbol_i^{\semisimple} $. In particular, each pair of subgroups in the products above have trivial intersection and therefore the products are semidirect, i.e. $ \widetilde{W}^i $ has the semidirect product decompositions $ \widetilde{W}^i = W_{\aff}^i \rtimes \Omega^i $ and $ \widetilde{W}^i = \Lambda^i \rtimes \finiteweylgroup^i $. It is then obvious also that $ \Lambda^i $ is the subgroup of all elements of $ \widetilde{W}^i $ which act by translations on $ \apartmentsymbol_i^{\semisimple} $.

In summary, for each $ i $, the quasi-Coxeter group $ \widetilde{W}^i = W_{\aff} ( \scaledrootsystemsymbol_i ) \rtimes \Omega^i $ acting on the apartment $ \apartmentsymbol_i^{\semisimple} $ matches the setup and satisfies the hypotheses of \cite{roro}--see \S2.1 and \S3.2 of \cite{roro} for more details and explanation of terminology.

It is now routine to extend the scope of \cite{roro} to \emph{reducible} root systems:
\begin{lemma} \label{Lextensionofdiamondpapertheorem}
If $ w \in \widetilde{W} $ and $ w \notin \Omega_M $ then there exist $ s_1, \ldots, s_n, s \in \Delta_{\aff} $ such that $ \ell ( s s_n \cdots s_1 w s_1 \cdots s_n s ) > \ell ( s_n \cdots s_1 w s_1 \cdots s_n ) = \cdots = \ell ( s_1 w s_1 ) = \ell ( w ) $.
\end{lemma}

\begin{proof}
Since $ W_{\aff} = W_{\aff} ( \scaledrootsystemsymbol_1 ) \times \cdots \times W_{\aff} ( \scaledrootsystemsymbol_r ) $, we can factor $ w $ as $ w = ( u_1, \ldots, u_r, \tau ) $ for some $ u_i \in W_{\aff} ( \scaledrootsystemsymbol_i ) $ and $ \tau \in \Omega_G $. Denote by $ \tau_i \in \Omega^i $ the restriction of the transformation $ \tau $ to $ \apartmentsymbol_i^{\semisimple} $ and set $ w_i \defeq ( u_i, \tau_i ) \in \widetilde{W}^i $. Since $ w \notin \Omega_M $ and since $ w $ acts on $ \apartmentsymbol^{\semisimple} = \apartmentsymbol_1^{\semisimple} \times \cdots \times \apartmentsymbol_r^{\semisimple} $ by the product $ ( w_1, \ldots, w_r ) $, there must exist $ i $ such that $ w_i \notin \Lambda^i $. By the discussion preceding this proof, we may apply the main theorem of \cite{roro} to the element $ w_i $ and conclude the existence of $ s_1, \ldots, s_n, s \in \Delta_{\aff}^i $ such that $ \ell_i ( s s_n \cdots s_1 w_i s_1 \cdots s_n s ) > \ell_i ( s_n \cdots s_1 w_i s_1 \cdots s_n ) = \cdots = \ell_i ( s_1 w_i s_1 ) = \ell_i ( w_i ) $. Since the permutation of $ \Delta_{\aff}^i $ by $ \Omega_G $ factors through $ \Omega^i $, since $ \Delta_{\aff}^i $ commutes with $ W_{\aff} ( \scaledrootsystemsymbol_j ) $ for all $ j \neq i $, and since $ \ell = \ell_1 + \cdots + \ell_r $, it follows immediately that $ \ell ( s s_n \cdots s_1 w s_1 \cdots s_n s ) > \ell ( s_n \cdots s_1 w s_1 \cdots s_n ) = \cdots = \ell ( s_1 w s_1 ) = \ell ( w ) $, as desired.
\end{proof}

For convenience of the reader, the application to Hecke algebras from \cite{roro} is reproduced here:
\begin{lemma} \label{Lextensionofdiamondpapercorollary}
Recall that for each $ \mathcal{O} \in \Omega_M / \finiteweylgroup $, $ \ell ( \lambda ) $ and $ \Omega_G ( \lambda ) $ are both independent of $ \lambda \in \mathcal{O} $. Denote the common values by $ \ell ( \mathcal{O} ) $ and $ \Omega_G ( \mathcal{O} ) $.

Suppose that for each $ \mathcal{O} \in \Omega_M / \finiteweylgroup $ a function $ \zeta_{ \mathcal{O} } \in Z [ \mathcal{H} ( G ; I ) ] $ is known such that $ \Omega_G ( w ) = \Omega ( \mathcal{O} ) $ and $ \ell ( w ) \leq \ell ( \mathcal{O} ) $ for each $ w $ supporting $ \zeta_{ \mathcal{O} } $.

If the collection $ \{ \zeta_{ \mathcal{O} } \}_{ \mathcal{O} \in \Omega_M / \finiteweylgroup } $ is linearly-independent then it necessarily spans $ Z [ \mathcal{H} ( G ; I ) ] $.
\end{lemma}

\begin{proof}
Because of the Iwahori-Matsumoto Presentation (Proposition \ref{Piwahorimatsumotopresentation}) for $ \mathcal{H} ( G ; I ) $, the linear system describing the $ \C $-subspace $ Z [ \mathcal{H} ( G ; I ) ] $ here is the same as the one appearing in \S8.2 of \cite{roro}, and the proof that $ \{ \zeta_{ \mathcal{O} } \} $ is a spanning set is identical to the proof given in \cite{roro}, using the extension Lemma \ref{Lextensionofdiamondpapertheorem} to reducible root systems in place of the original main theorem of \cite{roro}.
\end{proof}

\subsection{Proof that the central Bernstein functions are a basis} \label{SSbernfunctionsarebasis}

\begin{prop} \label{Pcentralbernsteinbasis}
$ \{ z_{ \mathcal{O} } \}_{ \mathcal{O} \in \Omega_M / \finiteweylgroup } $ is a basis for $ Z [ \mathcal{H} ( G ; I ) ] $.
\end{prop}

\begin{proof}
We first verify linear-independence. Suppose for contradiction that $ c_1 z_{ \mathcal{O}_1 } + \cdots + c_r z_{ \mathcal{O}_r } = 0 $ for some $ \mathcal{O}_1, \ldots, \mathcal{O}_r \in \Omega_M / \finiteweylgroup $ and $ c_1, \ldots, c_r \in \C^{\times} $. Let $ \lambda_1, \ldots, \lambda_r \in \Omega_M^{\dom} $ be the unique elements such that $ \lambda_i \in \mathcal{O}_i $ for all $ i $ (Lemma \ref{Ldominancefacts}(\ref{Ldomrotate})). Choose $ i $ so that $ \ell ( \lambda_i ) \geq \ell ( \lambda_j ) $ for all $ j $. By the corollary to Lemma \ref{Lhainesfact} and the fact that $ \iwahoriheckethetaelementsymbol_{\lambda} = \normalizediwahoriheckebasissymbol_{\lambda} $ whenever $ \lambda \in \Omega_M^{\dom} $, $ z_{\mathcal{O}_i} ( \lambda_i ) \neq 0 $. If $ j \neq i $ and $ z_{\mathcal{O}_j} ( \lambda_i ) \neq 0 $ then by Lemma \ref{Lhainesfact} there would be some $ \mu \in \mathcal{O}_j $ such that $ \lambda_i \prec \mu $, but this would imply $ \ell ( \lambda_i ) < \ell ( \mu ) $, which contradicts the choice of $ \lambda_i $ since constancy of $ \ell $ on $ \finiteweylgroup $-conjugacy classes (Lemma \ref{Ldominancefacts}(\ref{Ldomorbit})) means $ \ell ( \mu ) = \ell ( \lambda_j ) $. Therefore, $ z_{\mathcal{O}_j} ( \lambda_i ) = 0 $ for all $ j \neq i $ and evaluating both sides of the dependence relation at $ \lambda_i $ yields $ c_i = 0 $, a contradiction. Now we verify spanning. This is immediate by Lemma \ref{Lextensionofdiamondpapercorollary} above since linear independence has already been verified (the hypotheses on $ \ell ( w ) $ and $ \Omega_G ( w ) $ in \cite{roro} are satisfied due, respectively, to Lemma \ref{Lhainesfact} and Lemma \ref{Lomegacomponentsofthetas}).
\end{proof}

\begin{remark}
This Bernstein Isomorphism was also proven in the Appendix of \cite{haines2} using the Bushnell-Kutzko theory of types.
\end{remark}

\section{When are the parameters in the commutation relation ``alternating''?} \label{Scompactbernsteinrelation}

A basic question about the commutation relation (Proposition \ref{Pbernsteincommutationrelation}) is: \emph{Under what circumstances is it possible that $ \indexedparam{0}{s} \neq \indexedparam{1}{s} $?}

On one hand, the answer is known and concise: the parameters for $ \mathcal{H} ( G; I ) $ are the same as for the extended affine Hecke algebra $ \mathcal{H} ( \Psi_M, \Delta_{\aff}, \param ) $ and so \S2.4 of \cite{lusztig} says that $ \indexedparam{0}{s} \neq \indexedparam{1}{s} $ can happen only if $ \overline{\scaledrootsystemsymbol} = \scaledrootsystemsymbol $ is type $ C $. But on the other hand, the type of the scaled root system $ \scaledrootsystemsymbol $ corresponding to $ \Phi_{\aff} $ can be different from the type of the original root system $ \relativerootsystem $, and \S3.4 of \cite{gortz} relates the possibility of $ \indexedparam{0}{s} \neq \indexedparam{1}{s} $ to the lack of sufficiently ``short'' translations in $ \widetilde{W} $.

Because of this, it is interesting to see concretely how the translation subgroup $ \Omega_M $ (especially in comparison to $ \cochargroup ( A ) $) somehow compensates (or doesn't) for a change in type occuring in construction $ \Phi_{\aff} \leadsto \scaledrootsystemsymbol $.

This section includes an example of groups whose relative root system is type $ C $ resp. $ B $ but whose affine root system is type $ B $ (see \S\ref{SSexampleSU}) resp. $ C $ (see \S\ref{SSexampleSO}), together with the translation subgroups $ \Omega_M $ of their Iwahori-Weyl groups. These examples are suggested by the tables in \S1.4.6 of \cite{BT1} and in \cite{tits}.

In what follows, $ E / F $ is a ramified quadratic extension, $ R_{E/F} ( \mult ) $ is the usual restriction-of-scalars torus $ * \mapsto ( * \otimes_F E )^{\times} $, and $ R^1_{E/F} ( \mult ) \subset R_{E/F} ( \mult ) $ is the usual norm torus (kernel of the norm map). Symbols $ \chi_i $ always denote standard generators for the characters of a split torus which will be clear from context.

\subsection{Commutation relation parameters vs. affine hyperplanes} \label{SSbernparameters}

\begin{center}
\emph{This subsection summarizes the ``hyperplane interpretation'' from \cite{gortz} of the parameters $ \indexedparam{0}{s} $ and $ \indexedparam{1}{s} $. For convenience, assume that the reduced root system $ \scaledrootsystemsymbol $ is \emph{irreducible}.}
\end{center}

If $ \mathfrak{f} $ is a face of the base alcove $ \alcovesymbol $ then by definition the \emph{type} of $ \mathfrak{f} $ is the unique $ s \in \Delta_{\aff} $ such that $ s ( \mathfrak{f} ) = \mathfrak{f} $. If $ \mathfrak{f}^{\prime} $ is the face of any other alcove, then there is a unique face $ \mathfrak{f} $ of the base alcove $ \alcovesymbol $ and a unique $ w \in W_{\aff} ( \scaledrootsystemsymbol ) $ such that $ w ( \mathfrak{f} ) = \mathfrak{f}^{\prime} $ and by definition the \emph{type} of $ \mathfrak{f}^{\prime} $ is the type of $ \mathfrak{f} $. In this case, define $ \param ( \mathfrak{f} ) \defeq \param ( s ) $.

If $ \mathfrak{f} $ and $ \mathfrak{f}^{\prime} $ are two faces with types $ s $ and $ s^{\prime} $ and $ w \in \widetilde{W} ( \Psi_M ) $ is such that $ w ( \mathfrak{f} ) = \mathfrak{f}^{\prime} $ then $ s $ is $ w $-conjugate to $ s^{\prime} $ and therefore $ \param ( s ) = \param ( s^{\prime} ) $. The content of Lemma 3.4.1 in \cite{gortz} is that if $ \mathfrak{f} $ and $ \mathfrak{f}^{\prime} $ are two faces contained in the same affine hyperplane $ H $ then the types $ s, s^{\prime} $  are necessarily conjugate and therefore one may define $ \param ( H ) \defeq \param ( s ) $ for any affine hyperplane $ H $ and the type $ s $ of any face contained in $ H $. The content of Lemma 3.4.2(1) in \cite{gortz} is that if $ H $, $ H^{\prime} $, $ H^{\prime \prime} $ are ``consecutive'' parallel affine hyperplanes then the fact that $ s_{ H^{\prime} } ( H ) = H^{\prime \prime} $ implies that $ \param ( H ) = \param ( H^{\prime \prime} ) $.

Therefore, if $ \alpha \in \scaledrootsystemsymbol $ and if $ \mathfrak{H} $ is the set of all affine hyperplanes in $ \apartmentsymbol $ that are parallel to the nullspace $ H $ of $ \apartmentsymbol \stackrel{\alpha}{\longrightarrow} \R $ then the set $ \{ \param ( H^{\prime} ) \suchthat H^{\prime} \in \mathfrak{H} \} $ consists of at most two values. One value is $ \param ( H ) = \param ( s_{\alpha} ) $, and combining Lemma 3.4.2``(2)'' of \cite{gortz} with \S2.4 of \cite{lusztig} yields that the other value (if there is one) is $ \param ( H^{\prime} ) = \param ( \widetilde{s}_{\alpha} ) $, where $ H^{\prime} $ is the nullspace of $ \apartmentsymbol \stackrel{\alpha-1}{\longrightarrow} \R $. If there exists $ w \in \widetilde{W} $ such that $ w ( H ) = H^{\prime} $ then by the previous paragraph $ \param ( H ) = \param ( H^{\prime} ) $, and therefore $ \indexedparam{0}{s_{\alpha}} = \indexedparam{1}{s_{\alpha}} $.

\subsection{Example of ramified even-dimensional unitary groups} \label{SSexampleSU}

Let $ \phi $ be the hermitian form on $ E^6 $ defined by the ``anti-identity'' matrix, i.e. the form defined by the rule $ \phi ( ( x_1, x_2, x_3, x_4, x_5, x_6 ), ( y_1, y_2, y_3, y_4, y_5, y_6 ) ) = x_1 \overline{y}_6 + x_2 \overline{y}_5 + x_3 \overline{y}_4 + x_4 \overline{y}_3 + x_5 \overline{y}_2 + x_6 \overline{y}_1 $ and let $ G $ be the unitary group of $ \phi $.

Recall that the relative root system $ \relativerootsystem $ of $ G $ is type $ C_3 $. Recall from \S2.d.1 of \cite{PR3} that the affine root system $ \Phi_{\aff} $ of $ G $ is
$ \{ \pm \chi_i \pm \chi_j + \frac{1}{2} \Z, \pm 2 \chi_i + \Z \suchthat i, j \in \{ 1, 2, 3 \}, i \neq j \} $. It follows that the scaled root system $ \scaledrootsystemsymbol $ for $ \Phi_{\aff} $ is type $ B_3 $.

The discussion in \S\ref{SSbernparameters} suggests that whenever $ \alpha \in \relativerootsystem $ and $ r \in \R $ are such that the hyperplanes $ H_{\alpha} $ and $ H_{\alpha + r} $ are consecutive there should be an element $ w \in \widetilde{W} $ such that $ w ( H_{\alpha} ) = H_{\alpha + r} $. In fact, there is a translation $ t \in \widetilde{W} $ accomplishing this. By basic properties of apartments and knowledge of the above affine root system $ \Phi_{\aff} $, the existence of such a $ t $ will follow from the existence for each $ i = 1, 2, 3 $ of $ \lambda_i \in \Omega_M $ such that $ \langle \chi_i, \lambda_i \rangle_{\R} = \frac{1}{2} $, which we now verify directly.

The $ 6 $-dimensional diagonal subtorus $ T = R_{E/F} ( \mult ) \times R_{E/F} ( \mult ) \times R_{E/F} ( \mult ) $ is a maximal $ F $-torus in $ G $ and the $ 3 $-dimensional split subtorus $ A = \mult \times \mult \times \mult $ is a maximal $ F $-split subtorus of $ G $. Since $ G $ is quasi-split, $ M \defeq C_G ( A ) = T $, and to compute how $ \Omega_T $ acts on the apartment $ \apartmentsymbol \defeq \cochargroup ( A ) \otimes_{\Z} \R $, we need to compute the natural map $ \cochargroup ( A ) \rightarrow ( \cochargroup ( T )_{\inertia} )^{\frob} = \Omega_T $. One can compute that $ \cochargroup ( \mult ) \incarrow \cochargroup ( R_{E/F} ( \mult ) ) $ is identified as $ \Z \rightarrow \Z \oplus \Z $ by $ n \mapsto ( n, n ) $ and that $ \cochargroup ( R_{E/F} ( \mult ) ) \canarrow ( \cochargroup ( R_{E/F} ( \mult ) )_{\inertia} )^{\frob} $ is identified as $ \Z \oplus \Z \rightarrow \Z $ by $ ( m, n ) \mapsto m + n $. Therefore, the action $ \Omega_T \rightarrow \cochargroup ( A ) \otimes_{\Z} \R $ of $ \Omega_T $ via translations is identified as $ \Z^3 \rightarrow \R^3 $ by $ ( a, b, c ) \mapsto \left( \frac{a}{2}, \frac{b}{2}, \frac{c}{2} \right) $. In particular, for each $ i = 1, 2, 3 $, there exists $ \lambda_i \in \Omega_T $ such that $ \langle \chi_i, \lambda_i \rangle_{\R} = \frac{1}{2} $.

\subsection{Example of ramified even-dimensional orthogonal groups} \label{SSexampleSO}

Let $ Q $ be the quadratic form (with non-trivial Hasse invariant) on $ F^3 \oplus E \oplus F^3 $ described in \S1.16 of \cite{tits}. Let $ G $ be the special orthogonal group of $ Q $.

By \S1.16 of \cite{tits}, the relative root system is type $ B_3 $, while the affine root system $ \Phi_{\aff} $ is $ \{ \pm \chi_i \pm \chi_j + \Z, \pm \chi_i + \frac{1}{2} \Z \suchthat i, j \in \{ 1, 2, 3 \}, i \neq j \} $. It follows that the scaled root system $ \scaledrootsystemsymbol $ for $ \Phi_{\aff} $ is type $ C_3 $.

The discussion in \S\ref{SSbernparameters} suggests that for each $ i = 1, 2, 3 $ there should \emph{not} exist any $ \lambda_i \in \Omega_M $ such that $ \langle \chi_i, \lambda_i \rangle_{\R} = \frac{1}{2} $ because otherwise $ t_{\lambda_i} ( H_{\chi_i} ) = H_{ \chi_i - \frac{1}{2} } $ and therefore $ \widetilde{W} $ would be transitive on the family $ \mathfrak{H} $ of affine hyperplanes parallel to $ H_{\chi_i} $ (although for this particular group $ \indexedparam{0}{s} = \indexedparam{1}{s} $ anyway because the Hecke algebra parameters $ \param : \Delta_{\aff} \rightarrow \N $ are constant; see \S1.16 of \cite{tits}). We verify this directly.

Since elements of $ R^1_{E/F} ( \mult ) $ are $ F $-linear norm-preserving transformations of $ E $, it is clear from \S1.16 of \cite{tits} that $ G $ contains the diagonal subgroup $ T = \mult \times \mult \times \mult \times R^1_{E/F} ( \mult ) $ as a maximal $ F $-torus, and the subgroup $ A = \mult \times \mult \times \mult $ is a maximal $ F $-split torus.

Since (by \cite{tits}) $ G $ is quasi-split, $ M \defeq C_G ( A ) = T $, and to compute how $ \Omega_T $ acts on the apartment $ \apartmentsymbol \defeq \cochargroup ( A ) \otimes_{\Z} \R $, we need to compute the natural map $ \cochargroup ( A ) \rightarrow ( \cochargroup ( T )_{\inertia} )^{\frob} $. One can identify $ \cochargroup ( R^1_{E/F} ( \mult ) )_{\inertia} $ as $ \Z / 2 \Z $ and so $ \cochargroup ( A ) \rightarrow ( \cochargroup ( T )_{\inertia} )^{\frob} $ is identified as $ \Z^3 \longrightarrow \Z^3 \oplus \Z / 2 \Z $ by $ ( a, b, c ) \mapsto ( a, b, c, [0] ) $. It follows that the action $ \Omega_T \rightarrow \cochargroup ( A ) \otimes_{\Z} \R $ of $ \Omega_T $ on $ \apartmentsymbol $ by translations is identified as $ \Z^3 \oplus \Z / 2 \Z \rightarrow \R^3 $ by $ ( a, b, c, d \text{ mod $ 2 $} ) \mapsto ( a, b, c ) $. In particular, for each $ i = 1, 2, 3 $ there is no $ \lambda_i \in \Omega_T $ such that $ \langle \chi_i, \lambda_i \rangle_{\R} = \frac{1}{2} $.

\end{document}